\documentclass[a4paper, 11pt, twoside, openany]{article}
\usepackage{amsmath,amssymb,stmaryrd,mathtools,amsthm}
\usepackage[all]{xy}
\usepackage{fancyhdr}
\usepackage{footmisc}
\usepackage{hyperref}
\usepackage[margin=22mm]{geometry}
\usepackage[UKenglish]{datetime}
\usepackage{graphicx}
\usepackage{wrapfig}
\usepackage[usenames, dvipsnames]{color}
\usepackage{accents}
\usepackage{cases}
\usepackage{enumerate}
\usepackage{tikz-cd}
\usepackage{multirow}
\usepackage[normalem]{ulem}
\usepackage{algorithm}
\usepackage{algorithmic}
\usepackage{array}

\newcommand{\pp}[2]{\frac{\partial #1}{\partial #2}}

\newtheorem{definition}{Definition}[section]
\newtheorem{proposition}{Proposition}[section]
\theoremstyle{definition}
\newtheorem{remark}{Remark}[section]

\numberwithin{equation}{section}

\newcounter{savefootnote}
\newcounter{symfootnote}
\newcommand{\symfootnote}[1]{%
   \setcounter{savefootnote}{\value{footnote}}%
   \setcounter{footnote}{\value{symfootnote}}%
   \ifnum\value{footnote}>8\setcounter{footnote}{0}\fi%
   \let\oldthefootnote=\thefootnote%
   \renewcommand{\thefootnote}{\fnsymbol{footnote}}%
   \footnote{#1}%
   \let\thefootnote=\oldthefootnote%
   \setcounter{symfootnote}{\value{footnote}}%
   \setcounter{footnote}{\value{savefootnote}}%
}

\begin{document}
%
%
%
\begin{center}
{\Large An accurate SUPG-stabilized continuous Galerkin discretization for anisotropic heat flux in magnetic confinement fusion}
\end{center}
\vspace{-3mm}
\hrulefill
\begin{center}
{Golo A. Wimmer$^{1}$\symfootnote{correspondence to: gwimmer@lanl.gov}, Ben S. Southworth$^1$, Koki Sagiyama$^2$, Xian-Zhu Tang$^1$}\\
\vspace{2mm}
{\textit{$^1$Los Alamos National Laboratory}, \textit{$^2$Imperial College London}}\\
\vspace{4mm}
\today
\end{center}

\begin{abstract}
We present a novel spatial discretization for the anisotropic heat
conduction equation, aimed at improved accuracy at the high levels of
anisotropy seen in a magnetized plasma, for example, for magnetic
confinement fusion. The new discretization is based on a mixed
formulation, introducing a form of the directional derivative along
the magnetic field as an auxiliary variable and discretizing both the
temperature and auxiliary fields in a continuous Galerkin (CG)
space. Both the temperature and auxiliary variable equations are
stabilized using the streamline upwind Petrov-Galerkin (SUPG) method,
ensuring a better representation of the directional derivatives and
therefore an overall more accurate solution. This approach can be seen
as the CG-based version of our previous work (Wimmer, Southworth,
Gregory, Tang, 2024), where we considered a mixed discontinuous
Galerkin (DG) spatial discretization including DG-upwind
stabilization. We prove consistency of the novel discretization, and
demonstrate its improved accuracy over existing CG-based methods in
test cases relevant to magnetic confinement fusion. This includes a
long-run tokamak equilibrium sustainment scenario, demonstrating a 35\% and
32\% spurious heat loss for existing primal and mixed CG-based
formulations versus 4\% for our novel SUPG-stabilized discretization.
\end{abstract}
\section{Introduction}
Plasmas in magnetic confinement fusion exhibit extremely anisotropic
heat fluxes~\cite{jardin2010computational}, conducting heat
at a rate up to 10 orders of magnitude greater parallel to magnetic
field lines than perpendicular \cite{gunter2007finite}. If the
parallel heat flux is not represented accurately in numerical
simulations, the resulting increased cross-diffusion perpendicular to
magnetic field lines may lead to spuriously short energy confinement
times. One way to ensure such an accurate representation is to align
the computational mesh with magnetic flux surfaces
\cite{dudson2015bout++, hoelzl2021jorek, jardin2010computational};
however, this may not be possible in more complex magnetohydrodynamic
(MHD) simulations including magnetic islands or even stochastic
magnetic field lines \cite{freidberg1987ideal,
  sovinec2004nonlinear}. In practice, such simulations therefore
require numerical methods that accurately represent the parallel heat
flux independently of mesh alignment.

Due to the importance of anisotropic heat flux in extended MHD and
two-fluids simulations for magnetic fusion plasmas, a range of
numerical methods have been developed in recent years in order to
tackle this challenge. This includes, but is not limited to, improved
accuracy through higher polynomial order
discretizations~\cite{green2022efficient, green2024efficient}, mesh
refinement at areas of large expected cross-diffusion
errors~\cite{vogl2022mesh}, asymptotic preserving methods (and
extensions thereof), aimed at well-posedness in the limit of infinite
anisotropy strength~\cite{degond2012asymptotic, deluzet2019two,
  jin1999efficient, narski2014asymptotic, yang2019preserving,
  yang2024accuracy}, and first-order hyperbolic system methods
including auxiliary variables for the temperature's gradient
components together with pseudo-time advancement
terms~\cite{chamarthi2019first} (and references therein). Another
auxiliary variable-based approach is presented in \cite{gunter2005fd},
where the auxiliary variable represents the temperature's directional
derivative along magnetic field lines. The same idea is expanded on in
our previous work~\cite{wimmer2024fast}, where a mixed, discontinuous Galerkin
(DG) finite element-based method is derived based on classical
DG-upwinding as frequently employed for advection terms in
hydrodynamic problems. In particular, we showed the directional
gradient's improved representation resulting from transport
stabilization leads to greatly improved accuracy when compared to its
non-stabilized counterpart from \cite{gunter2005fd} in test cases with
extreme anisotropy.

Many of the aforementioned approaches are based on the finite element
method, which is often considered in magnetic confinement fusion
simulations due to, for instance, their flexibility in meshing complex
geometries. Next to the aforementioned DG-based work
in~\cite{green2022efficient, green2024efficient, wimmer2024fast}, this
includes additional DG-based approaches such as hybrid
DG~\cite{giorgiani2020high} and local
DG~\cite{held2016three}. However, DG-based methods may lead to
challenges such as an increased number of degree of freedoms (DOFs),
as well as an increased complexity in terms of implementation when
compared to continuous Galerkin (CG) methods. Further, DG-based
implementations often require interior penalty formulations, leading
to a potential loss in the order of accuracy. For these reasons, and
since magnetic confinement fusion scenarios generally involve a
shock-free low Mach number regime, many finite element-based MHD codes
used to simulate such regimes, such as~\cite{bonilla2023fully,
  breslau2009some, hoelzl2021jorek, sovinec2004nonlinear}, deploy CG
(or higher order regularity Galerkin) rather than DG methods.

In this paper, we present an extension of our previous DG-based
upwind-stabilized work in \cite{wimmer2024fast} to CG-based
discretizations. As in \cite{wimmer2024fast}, we consider a mixed
formulation of the anisotropic heat flux, using a form of the
temperature's directional gradient along magnetic field lines as the
auxiliary variable. Analogously to upwinding in the DG case, we apply
a form of the Streamline Upwind Petrov-Galerkin (SUPG)
method~\cite{brooks1982streamline} for the two directional derivatives
comprising the anisotropic diffusion operator. We motivate our
particular choice of SUPG formulation and further demonstrate
consistency of the resulting spatial discretization for the
anisotropic heat flux equation. Additionally, in a series of numerical
tests, we validate its order of accuracy as well as superior
performance in realistic test cases when compared to a primal weak
formulation and the mixed formulation of \cite{gunter2005fd}. In
particular, we study the spread of a temperature perturbation in a
2D magnetic flux surface, an MHD equilibrium in a full-torus tokamak
domain including a temperature-dependent parallel conductivity
coefficient, as well as the spread of a temperature perturbation
through a flux tube in the latter domain. This includes a long run full-torus
tokamak equilibrium scenario, demonstrating a 35\% and 32\% spurious
heat loss for existing primal and mixed CG-based formulations versus
4\% for our novel SUPG-stabilized discretization. Further, the tests
contain parameter scales suitable for MHD models, including the choice
of time step and parallel conductivity.

The remainder of this paper is structured as follows: In
Section~\ref{sec_background}, we review the anisotropic heat
conduction equation and include a brief description of standard
CG-based primal and mixed space discretizations. In
Section~\ref{sec_novelty}, we introduce our novel SUPG-based scheme,
including a discussion on its derivation as well as consistency and
the choice of SUPG stabilization parameter. In
Section~\ref{sec_Numerical_results}, we present and discuss numerical
results. Finally, in Section \ref{sec_conclusion}, we review our
results and discuss possible future work.
%
\section{Background} \label{sec_background}
In this section, we briefly review the anisotropic heat flux
equation. Further, we recall existing continuous Galerkin-based
methods for the equation's spatial discretization, with which we will
compare our novel method in the numerical results section.
%

\textbf{Anisotropic heat flux equation.} We consider the evolution of temperature $T$ within a plasma domain $\Omega$ as governed by a heat flux $\mathbf{q}$ and forcing $S$ according to
\begin{subequations}\label{T_eqn_orig}
\begin{align}
\pp{T}{t} - \nabla \cdot \mathbf{q} & = S,  \\
\mathbf{q} &= \kappa_{\parallel} \nabla_{\parallel} T + \kappa_{\perp} \nabla_{\perp} T, \label{q_orig}
\end{align}
\end{subequations}
for parallel and perpendicular heat conductivities $\kappa_\parallel$,
$\kappa_\perp$, respectively. The directional gradients are given by
\begin{equation}
\nabla_\parallel T= \mathbf{b} \big(\mathbf{b} \cdot \nabla
T\big), \hspace{1cm} \nabla_\perp T = \nabla T - \nabla_\parallel T,
\end{equation}
where $\mathbf{b} = \mathbf{B} / |\mathbf{B}|$ denotes the normalized
magnetic vector field determining the direction of anisotropy. In the
following, we consider a reformulated version of $\mathbf{q}$ that
contains an isotropic and a purely anisotropic part
\begin{equation}
\mathbf{q} = \kappa_\Delta \mathbf{b}(\mathbf{b} \cdot \nabla T) + \kappa_\perp \nabla T,
\end{equation}
where $\kappa_\Delta = \kappa_\parallel - \kappa_\perp$. Additionally, the equation is equipped with Dirichlet boundary conditions
\begin{equation}
T(\mathbf{x}, t) = T_{bc}(\mathbf{x}), \hspace{1cm} \mathbf{x} \in \partial \Omega, \label{Dirichlet_bc}
\end{equation}
where in magnetic confinement fusion, the domain's boundary is given
by the device's plasma-facing wall. Last, the ``forcing'' term $S$ contains
additional terms arising in the context of magnetohydrodynamic models
(see \eqref{p_eqn_full} in Appendix \ref{App_kappas}). In this
paper, we will for simplicity assume $S$ to be a known field, which is
a reasonable assumption for instance in view of split-time
discretizations in magnetohydrodynamics, where different dynamics --
including the heat flux -- are computed at separate stages.

Finally, we note that in practice, the heat conductivities
$\kappa_\Delta$, $\kappa_\perp$ may depend on e.g. the temperature,
density, and magnetic field, rendering the anisotropic heat conduction
problem nonlinear. For instance, under collisional closures of
Braginskii~\cite{braginskii1965transport}, the parallel conductivity
scales as
\begin{equation}
\kappa_\parallel \propto T^{5/2}; \label{kappa_par_scale}
\end{equation}
for details see Appendix \ref{App_kappas}. In the numerical results
section below, we will consider such temperature dependent
conductivities for one of the test cases, and use constant in time
coefficients otherwise.

%
\textbf{Existing spatial discretizations.} A straightforward
discretization for the anisotropic heat conduction
equation~\eqref{T_eqn_orig} can be derived by considering $T_h$ in a
$k^{th}$ polynomial order continuous Galerkin finite element space $\mathbb{V}^{\text{CG}}_k$
equipped with the Dirichlet boundary conditions \eqref{Dirichlet_bc}
\begin{equation}
T_h \in \mathbb{V}^{\text{CG}}_{k, bc} = \{\gamma \in \mathbb{V}^{\text{CG}}_k, \; \gamma |_{\partial \Omega} = T_{bc}\}, \\
\end{equation}
and discretizing the elliptic operators weakly.  This leads to
\begin{align}
\left\langle \gamma, \pp{T_h}{t} \right\rangle + \langle \mathbf{b} \cdot \nabla \gamma, \kappa_\Delta \mathbf{b} \cdot \nabla T_h \rangle + \langle \nabla \gamma, \kappa_\perp \nabla T_h \rangle = \langle \gamma, S\rangle && \forall \gamma \in \mathring{\mathbb{V}}^{\text{CG}}_k, \label{CG_straight}
\end{align}
where $\langle \cdot, \cdot \rangle$ denotes the $L^2$-inner product, and 
\begin{equation}
\mathring{\mathbb{V}}^{\text{CG}}_k = \{\gamma \in \mathbb{V}^{\text{CG}}_k, \; \gamma|_{\partial \Omega} = 0\}.
\end{equation}
At higher levels of anisotropy, the accuracy of \eqref{CG_straight} can deteriorate drastically due to an inadequate representation of the directional gradient $\mathbf{b} \cdot \nabla T$. This leads to a spurious leak of parallel heat flux in the perpendicular direction, and in order to avoid this, \cite{gunter2005fd} introduced an auxiliary variable to better represent the directional gradient. Given $T_h \in \mathbb{V}^{\text{CG}}_{k, bc}$, the scaled directional gradient $\zeta \coloneqq \sqrt{\kappa_\Delta}\mathbf{b} \cdot \nabla T$ can be argued to be a function in the discontinuous Galerkin space $\mathbb{V}^{\text{DG}}_{k-1}$, and the resulting mixed discretization is given by
\begingroup
\addtolength{\jot}{4mm}
\begin{subequations} \label{CG_mixed_Gunter}
\begin{align}
&\left\langle \gamma, \pp{T_h}{t} \right\rangle + \langle \mathbf{b} \cdot \nabla \gamma, \sqrt{\kappa_\Delta} \zeta_h \rangle + \langle \nabla \gamma, \kappa_\perp, \nabla T_h \rangle = \langle \gamma, S\rangle & \forall \gamma \in \mathring{\mathbb{V}}^{\text{CG}}_k, \label{CG_mixed_Gunter_T} \\
&\left\langle \phi, \zeta_h - \sqrt{\kappa_\Delta} \mathbf{b} \cdot \nabla T_h \right\rangle = 0 & \forall \phi \in \mathbb{V}^{\text{DG}}_{k-1}.
\end{align}
\end{subequations}
\endgroup

Due to the
ill conditioning (and corresponding time-step restriction for explicit integration) of discretized second derivatives, 
the anisotropic heat flux equation is typically discretized implicitly in time. In our numerical results section, we consider a midpoint rule in time for \eqref{CG_straight}, \eqref{CG_mixed_Gunter} as well as our novel space discretization to be introduced in the next section, although other more general implicit integration schemes could be considered. For scenarios with temperature dependent conductivity coefficients, this leads to a nonlinear system of equations, which can be solved using a Newton iteration method, or alternatively e.g., a quasi-Newton iteration approach with known estimates for the temperature field appearing in the coefficients.
%
\section{Discretization} \label{sec_novelty}
Having introduced the anisotropic heat flux equation as well as standard CG-based spatial discretizations thereof, we next present our novel upwind-stabilized method. For this purpose, we start with a brief motivation leading up to our method's definition, followed by a discussion of its properties.

\textbf{Motivation.} In \cite{wimmer2024fast}, we built on the
observation of \cite{gunter2005fd} that a better discrete
representation of the directional gradient $\mathbf{b} \cdot \nabla T$
is a key property to avoid spurious cross-diffusion. To improve this
representation, we constructed a DG-based version of the mixed scheme
\eqref{CG_mixed_Gunter}, using classical DG upwinding
\cite{kuzmin2010guide} to discretize the directional gradient. In
analogy, for CG-based schemes, one could consider an upwind-type
approach such as the Streamline Upwind Petrov Galerkin (SUPG)
method~\cite{brooks1982streamline}, where first, an additional term is
included to introduce diffusion in the direction of streamlines --
thereby targeting numerical noise introduced by advection. Second,
additional terms are included to ensure a consistent method. Put
together, the two changes can be achieved by considering a modified
test function space, leading to a Petrov-Galerkin method where test
and trial function spaces are different. For a simple transport
equation discretized in a CG space, this modification corresponds to
\begin{align}
\left\langle \gamma, \pp{\theta_h}{t} + \mathbf{u} \cdot \nabla \theta_h \right\rangle = 0  \;\;\; \overset{\text{SUPG}}{\longrightarrow} \;\;\; \left\langle \gamma + \tau \mathbf{u} \cdot \nabla \gamma, \pp{\theta_h}{t} + \mathbf{u} \cdot \nabla \theta_h \right\rangle = 0 &&\forall \gamma \in \mathring{\mathbb{V}}_k, \label{SUPG_adv}
\end{align}
for discrete tracer field $\theta_h$, advecting velocity field $\mathbf{u}$, and SUPG stabilization parameter $\tau$. In particular, the SUPG-modified equation is clearly still strongly consistent, in the sense that a solution $\theta$ to the strong equation $\pp{\theta}{t} + \mathbf{u} \cdot \nabla \theta = 0$ is also a solution to the SUPG weak form. Further, setting $\gamma = \theta_h$, we find a stabilizing effect
\begin{equation}
\frac{1}{2}\frac{d}{dt}\|\theta_h\|_2^2 = \dots - \|\sqrt{\tau} \mathbf{u} \cdot \nabla \theta_h \|_2^2.
\end{equation}

While upwind stabilization is well-motivated in hyperbolic problems, its use for the anisotropic diffusion equation is less immediately clear, as the latter equation is parabolic. To discuss this further, we note that the algebraic system of equations for the (implicitly discretized in time) mixed method \eqref{CG_mixed_Gunter} can be written as
\begin{equation}
\begin{pmatrix}
\frac{1}{\delta t} M_{CG} + L_\perp & D_\parallel^T \\
-D_\parallel & M_{DG}
\end{pmatrix}
\begin{pmatrix}
\mathbf{T}_c\\
\boldsymbol{\zeta}_c
\end{pmatrix}
=
\begin{pmatrix}
R_T \\
R_\zeta
\end{pmatrix}, \label{matrix_Gunter}
\end{equation}
for time step $\delta t$, CG and DG mass matrices $M_{CG}$, $M_{DG}$, respectively, discrete directional gradient $D_\parallel$ corresponding to $\sqrt{\kappa_\Delta} \mathbf{b} \cdot \nabla$ and discrete Laplacian $L_\perp$ weighted by $\kappa_\perp$, where strong Dirichlet boundary conditions for $T_h$ are applied appropriately to the matrix blocks. Further, we have degree of freedom coefficient vectors $\mathbf{T}_c$, $\boldsymbol{\zeta}_c$ to be solved for, and known right-hand side vectors $R_T$, $R_\zeta$. Additionally, we ignored terms occurring in the case of temperature-dependent conductivity coefficients. Eliminating the auxiliary variable, we obtain
\begin{equation}
\left(\frac{1}{\delta t}M_{CG} + D^T M_{DG}^{-1} D + L_\perp\right) \mathbf{T}_c = \dots, \label{Schur_Gunter}
\end{equation}
that is we recover a symmetric positive definite matrix $M_{CG} + D_\parallel^T M_{DG}^{-1} D_\parallel + L_\perp$ representing the parabolic problem, where $D_\parallel^T M_{DG}^{-1} D_\parallel$ denotes a discrete form of a strictly anisotropic diffusion operator along $\mathbf{b}$ with diffusion coefficient $\kappa_\Delta$. For the DG-upwind scheme introduced in \cite{wimmer2024fast}, the corresponding system of equations can be written as
\begin{equation}
\begin{pmatrix}
  \frac{1}{\delta t}M_{DG} + L_{DG,\perp} + M_{BC} && G_\parallel^T \\
  -G_\parallel && M_{DG}
\end{pmatrix}
\begin{pmatrix}
  \mathbf{T}_c \\
  \boldsymbol{\zeta}_c
\end{pmatrix}
=
\begin{pmatrix}
  R_T \\
  R_\zeta
\end{pmatrix}, \label{matrix_DG}
\end{equation}
where $L_{DG, \perp}$ is a suitable (symmetric) DG-based discretization of the Laplacian weighted by $\kappa_\perp$, and $G_\parallel$ is a DG-upwinding-based discretization of the directional gradient $\sqrt{\kappa_\Delta} \mathbf{b} \cdot \nabla$. Additionally, $M_{BC}$ is a boundary penalty term used to enforce the Dirichlet boundary conditions weakly. Similarly to \eqref{Schur_Gunter}, upon eliminating the auxiliary variable, we recover an operator for $\mathbf{T}_c$ that was shown in \cite{wimmer2024fast} to be symmetric positive definite. In other words, while upwinding is used for a better representation of the directional gradients occurring in the mixed formulation for the anisotropic diffusion equation, the overall method still represents an elliptic operator. With this in mind, we aim to obtain a similar result for CG-based methods including SUPG-stabilization for the directionaly gradients.

\textbf{SUPG-based method.} We next discuss our novel CG-based discretization, aiming for a stabilized form of the directional gradients while maintaining consistency and keeping in mind the discrete parabolic operator property discussed above. For this purpose, we first define suitable streamline upwind operators.
%
\begin{definition} \label{def_cg_SUPG_ops}
For magnetic field $\mathbf{B}$ and conductivity difference $\kappa_\Delta = \kappa_\parallel - \kappa_\perp$, we set the upwind velocity $\mathbf{s} \coloneqq \sqrt{\kappa_\Delta(T)} \mathbf{B}/|\mathbf{B}|$. For arbitrary functions $\chi$, $\upsilon$, the advective and flux-based SUPG-modified bilinear forms $M_a$, $M_f$ are defined by
\begingroup
\addtolength{\jot}{4mm}
\begin{subequations} \label{SUPG_operators}
\begin{align}
&M^a(\chi,\upsilon) \coloneqq \langle \chi + \tau \mathbf{s} \cdot \nabla \chi, \upsilon \rangle, \\
&M^f(\chi, \upsilon) \coloneqq \langle \chi, \upsilon + \mathbf{s} \cdot \nabla (\tau \upsilon) \rangle. \label{def_Mf}
\end{align}
\end{subequations}
\endgroup
Further, for arbitrary functions $\omega$, $\upsilon$, the SUPG-modified bilinear form $G^f_\parallel$ is defined by
\begin{equation}
G^f_\parallel(\omega, \upsilon) \coloneqq M^f(\mathbf{s} \cdot \nabla \omega, \upsilon).
\end{equation}
\end{definition}
Note that $M^a$, $M^f$, and $G^f_\parallel$ are well-defined after discretization if $\chi$, $\upsilon$, $\omega$ are functions of a CG space, and further provided that the stabilization parameter $\tau$ is continuous. After discretization, $M^a$ and $M^f$ both are forms corresponding to SUPG-modified mass matrices, and they reduce to a mass matrix as $\tau \rightarrow 0$. Further, $G^f_\parallel$ corresponds to an SUPG-modified directional gradient.

\begin{definition} \label{def_cg_SUPG}
The SUPG-modified, CG-based mixed discretization of the anisotropic heat flux equation \eqref{T_eqn_orig} for velocity $\mathbf{s} \coloneqq \sqrt{\kappa_\Delta(T)} \mathbf{B}/|\mathbf{B}|$ is given by finding $(T_h, \zeta_h) \in \mathbb{V}^{\text{CG}}_{k, bc} \times \mathbb{V}^{\text{CG}}_k$ such that
\begingroup
\addtolength{\jot}{4mm}
\begin{subequations} \label{SUPG_method}
\begin{align}
&M^a\left(\!\gamma, \pp{T_h}{t}\!\right) + G^f_\parallel(\gamma, \zeta_h) + \langle \nabla \gamma, \kappa_\perp \nabla T_h \rangle \!-\! \langle \tau \mathbf{s} \cdot \nabla \gamma, \nabla \cdot (\kappa_\perp \nabla T_h) \rangle \!=\! M^a(\gamma, S) & \forall \gamma \in \mathring{\mathbb{V}}^{CG}_k, \label{SUPG_method_T}\\
&M^f(\eta, \zeta_h) - G^f_\parallel(T_h, \eta) + \int_{\partial\Omega} \tau \eta (\mathbf{s} \cdot \nabla T_h)(\mathbf{n} \cdot \mathbf{s})dS = 0 &\forall \eta \in \mathbb{V}^{CG}_k, \label{SUPG_method_zeta}
\end{align}
\end{subequations}
\endgroup
for outward unit normal vector $\mathbf{n}$.
\end{definition}

\begin{proposition} \label{prop_CG_consistent}
The mixed discretization \eqref{SUPG_method} is consistent, in the sense that a strong solution $T$ to the non-discretized equation \eqref{T_eqn_orig} satisfies \eqref{SUPG_method} with a sufficiently regular test function space $\mathbb{V}_T \supset \mathbb{V}_k^{CG}$ and parameters $\kappa_\parallel$, $\kappa_\perp$, $\tau$ and $\mathbf{B}$.
\end{proposition}
\begin{proof}
We want to show that \eqref{SUPG_method} holds true with $T_h$ set to a strong solution $T$. For this purpose, we first eliminate $\zeta_h$ by using the definition of $G^f_\parallel$ and applying integration by parts in $G^f_\parallel(T_h, \eta)$, which leads to a reformulation of \eqref{SUPG_method_zeta} of the form
\begin{equation}
M^f(\eta, \zeta_h) = \langle \mathbf{s}\cdot \nabla T, \eta + \mathbf{s} \cdot \nabla (\tau \eta) \rangle - \int_{\partial\Omega} \tau \eta (\mathbf{s} \cdot \nabla T)(\mathbf{n} \cdot \mathbf{s})dS = \left\langle \eta, \mathbf{s}\cdot\nabla T - \tau \nabla \cdot \big(\mathbf{s}(\mathbf{s}\cdot \nabla T)\big)\right\rangle, \label{zeta_eq_reformulate}
\end{equation}
where we assumed sufficient regularity for the test functions $\eta$, as well as $T$, $\mathbf{B}$ and $\tau$. Next we set 
$\eta = \mathbf{s}\cdot \nabla \gamma$ in \eqref{zeta_eq_reformulate} for test function $\gamma$ appearing in the temperature equation (again assuming sufficient regularity), which yields
\begin{equation}
G_\parallel^f(\gamma, \zeta_h) = M^f(\mathbf{s} \cdot \nabla \gamma, \zeta) = \left\langle \mathbf{s} \cdot \nabla \gamma, \mathbf{s}\cdot\nabla T - \tau \nabla \cdot \big(\mathbf{s}(\mathbf{s}\cdot \nabla T)\big)\right\rangle.
\end{equation}
We can then substitute this in the second term in \eqref{SUPG_method_T}, from which we obtain
\begingroup
\addtolength{\jot}{2mm}
\begin{align}
&M^a\!\left(\!\gamma, \!\pp{T}{t}\right) + \left\langle \mathbf{s} \cdot \nabla \gamma, \mathbf{s}\cdot\nabla T - \tau \nabla \cdot \big(\mathbf{s}(\mathbf{s}\cdot \nabla T)\big)\right\rangle \nonumber \\
&\hspace{2cm} + \langle \nabla \gamma, \kappa_\perp \nabla T \rangle - \langle \tau \mathbf{s} \cdot \nabla \gamma, \nabla \cdot (\kappa_\perp \nabla T) \rangle = M^a(\gamma, S) & \forall \gamma \in \mathring{\mathbb{V}}_T \label{SUPG_method_T_I}.
\end{align}
\endgroup
Finally, collecting terms in \eqref{SUPG_method_T_I}, we obtain
\begingroup
\addtolength{\jot}{2mm}
\begin{align}
\left\langle \gamma, \pp{T}{t} \right\rangle + &\langle \mathbf{s} \cdot \nabla \gamma, \mathbf{s} \cdot \nabla T \rangle + \langle \nabla \gamma, \kappa_\perp \nabla T \rangle - \langle \gamma, S \rangle = \nonumber \\
&- \left\langle \tau \mathbf{s} \cdot \nabla \gamma, \pp{T}{t} - \nabla \cdot \big(\mathbf{s}(\mathbf{s} \cdot \nabla T)\big) - \nabla \cdot (\kappa_\perp \nabla T) - S \right\rangle,
\end{align}
\endgroup
for any test function $\gamma \in \mathring{\mathbb{V}}_T$. This equation is satisfied by the strong solution $T$ as required, since the left-hand side is the standard weak formulation of the anisotropic heat flux equation \eqref{T_eqn_orig}, while the right-hand side consists of the inner product of \eqref{T_eqn_orig} in strong form together with $\tau \mathbf{s} \cdot \nabla \gamma$.
\end{proof}
We complete this section with a brief discussion and remarks on the spatial discretization \eqref{SUPG_method}.

\textbf{Ellipticity.} First, given an implicit discretization in time, the SUPG-modified method \eqref{SUPG_method} can be written as
\begin{equation}
\begin{pmatrix}
\frac{1}{\delta t} M^a + L^f_\perp & {G^f_\parallel}^T \\
-G^f_\parallel - G^f_{\parallel, b} & M^f
\end{pmatrix}
\begin{pmatrix}
\mathbf{T}_c\\
\boldsymbol{\zeta}_c
\end{pmatrix}
=
\begin{pmatrix}
R_T \\
R_\zeta
\end{pmatrix}, \label{matrix_SUPG}
\end{equation}
where in a slight abuse of notation, here $M^a$, $M^f$, and $G^f_\parallel$ denote the matrix version of the corresponding forms described in Definition \eqref{def_cg_SUPG_ops}, and $L^f_\perp$ corresponds to the last two terms on the left-hand side of \eqref{SUPG_method_T}. Further, $G^f_{\parallel, b}$ denotes a boundary-type directional gradient matrix corresponding to the (negative version of the) last term on the left-hand side of \eqref{SUPG_method_zeta}. Eliminating the auxiliary variable as before, we obtain an algebraic form of the heat equation given by
\begin{equation}
\left(\frac{1}{\delta t}M^a + {G^f_\parallel}^T \left(M^f\right)^{-1} \left(G^f_\parallel + G^f_{\parallel,b}\right) + L^f_\perp\right) \mathbf{T}_c = \dots, \label{Schur_SUPG}
\end{equation}
which retains the the same ``transpose gradient''-``mass inverse''-``gradient'' structure for the anisotropic heat flux term as in \eqref{Schur_SUPG}, up to the SUPG-modification in $M^f$ and the additional SUPG-related term $G^f_{\parallel, b}$. This loss in ellipticity of the anisotropic diffusion operator -- and also the SUPG-modified isotropic operator $L^f_\perp$ -- is a result of ensuring the discretization to be consistent by reverting to a Petrov-Galerkin method. Note that as $\tau \rightarrow 0$, $M^f$ reduces to the usual mass matrix, $G^f_{\parallel,b}$ vanishes, and $L^f_\perp$ reduces to $L_\perp$, and we therefore recover the parabolic structure in this limit. As usual for SUPG methods, in practice the Petrov-Galerkin type modification translates to potential instabilities if the SUPG stabilization parameter $\tau$ is not chosen carefully. The loss of symmetry and ellipticity seems like a downside from a linear solver perspective, but for very high anisotropy, state-of-the-art elliptic solvers like geometric and algebraic multigrid methods are not robust \cite{wimmer2024fast}.

\textbf{Stabilization parameter.} For the choice of stabilization parameter, we recall that the SUPG-modification to test functions $\gamma$ is given by $\tau \mathbf{s} \cdot \nabla \gamma$, for $\mathbf{s} = \sqrt{\kappa_\Delta(T)} \mathbf{B}/|\mathbf{B}|$. The heat conductivity's unit is m$^2$/s, and therefore for the overall SUPG modification to be non-dimensional, we require the unit of $\tau$ to be s$^{1/2}$. This is different to the standard SUPG method for advection equations such as \eqref{SUPG_adv}, where the streamline velocity $\mathbf{u}$ is in m/s and $\tau$ represents a time scale in seconds. The difference follows since here we do not consider a hyperbolic equation. If for instance instead we considered a discretization for a mixed form of an anisotropic wave equation -- which could be achieved e.g. by replacing $\zeta_h$ by $\pp{\zeta_h}{t}$ in \eqref{SUPG_method_zeta} and taking $\sqrt{\kappa_\Delta}$ as a wave speed with unit m/s --  $\tau$ would scale in seconds again. Guided by common choices of $\tau$ in the computational fluid dynamics literature \cite{tezduyar2002stabilization}, in the numerical results section below we set
\begin{equation}
\tau = \left(\frac{2}{\sqrt{\delta t}} + \frac{k \sqrt{\kappa_\Delta}}{\delta x}\right)^{-1}, \label{def_tau}
\end{equation}
for local mesh cell length scale $\delta x$ and polynomial degree $k$ of the CG function space $T_h$ and $\zeta_h$ are discretized in. Note that as we refine the mesh and time step, the stabilization contribution gets smaller, reaching $\tau = 0$ in the space-time refined limit. In this case, the mixed SUPG-modified method \eqref{def_cg_SUPG} reduces to the standard mixed method \eqref{CG_mixed_Gunter} up to the choice of finite element space for the auxiliary variable $\zeta_h$. In \eqref{CG_mixed_Gunter} the auxiliary variable is placed in a DG space in view of $\zeta_h$ as a gradient of the CG field $T_h$; here we placed $\zeta_h$ in a CG space since we apply SUPG -- which requires CG spaces -- to both the $T_h$ and $\zeta_h$ equations. Finally, as mentioned below Definition \ref{def_cg_SUPG_ops}, we note that $\tau$ occurs within gradient operations in our SUPG-modified mass matrix $M^f$, and therefore needs to be continuous for $M^f$ to be well-defined. In practice, $\tau$ is evaluated as an expression according to \eqref{def_tau} on quadrature points, including the local cell length scale $\delta x$, which we found to work well.

Additional adjustments to $\tau$ could for instance include information on the mesh relative to the magnetic field. This is of relevance for tokamak applications with frequently very anisotropic cells (see Figure \ref{tokamak_setup} in the numerical results section for a mildly anisotropic case); if a magnetic field line passes through such a cell along the elongated direction, we would expect the parallel diffusion to be significantly better resolved than if the field line passes through the short direction. More generally, an alignment of conductivity and mesh anisotropies may lead to a discrete system of equations that is overall more isotropic in character, and this can be addressed on the level of our stabilization paremeter e.g., by replacing $\delta x$ in \eqref{def_tau} by an average distance covered by $\mathbf{B}$ passing through the given cell.

\begin{remark}[\textit{Alternative SUPG-based forms}]
In the original mixed formulation \eqref{CG_mixed_Gunter}, the auxiliary variable $\zeta$ directly corresponds to $\mathbf{s} \cdot \nabla T$. For the SUPG-modified method \eqref{SUPG_method}, this no longer holds true, and more generally, there is some flexibility with the exact setup of our SUPG modification and definition of auxiliary variable $\zeta$. Simpler formulations include e.g. a more straight-forward application of SUPG-based on the original mixed form \eqref{CG_mixed_Gunter}, leading to
\begingroup
\addtolength{\jot}{4mm}
\begin{subequations} \label{CG_mixed_SUPG_simple}
\begin{align}
&\left\langle \gamma, \pp{T_h}{t} \right\rangle + \langle \mathbf{s} \cdot \nabla \gamma, \zeta_h \rangle + \langle \nabla \gamma, \kappa_\perp, \nabla T_h \rangle = \langle \gamma, S\rangle - \left\langle \tau \mathbf{s} \cdot \nabla \gamma, T_{rhs} \right\rangle & \forall \gamma \in \mathring{\mathbb{V}}^{\text{CG}}_k, \\
&\left\langle \eta + \tau \mathbf{s} \cdot \nabla \eta, \zeta_h - \mathbf{s} \cdot \nabla T_h \right\rangle = 0 & \forall \eta \in \mathbb{V}^{\text{CG}}_k,
\end{align}
\end{subequations}
\endgroup
were $T_{rhs}$ corresponds to the strong form \eqref{T_eqn_orig} of the anisotropic heat flux equation applied to the discrete field $T_h$. While this form is clearly consistent -- $T_{rhs}$ vanishes for strong solutions $T$ and $\zeta$ still corresponds to $\mathbf{s} \cdot \nabla T$ in strong form -- we found it to lead to worse results than the version defined in Definition \ref{def_cg_SUPG}, possibly because it can be shown to not lead to an approximately antisymmetric off-diagonal block structure as seen in \eqref{matrix_Gunter}, \eqref{matrix_DG} and \eqref{matrix_SUPG}. Finally, we also did not find any SUPG-based strategies for the primal form of the anisotropic heat equation -- that is without the use of an auxiliary variable -- that performed well in practical test cases.
\end{remark}

\begin{remark}[\textit{Solver considerations}] 
While we focus on spatial discretizations in this work, we end this section with a short remark on solver considerations. One possible way to solve for mixed systems such as \eqref{matrix_Gunter} and \eqref{matrix_SUPG} is to eliminate the auxiliary variable in order to obtain a Schur complement in the temperature field $T$, given by \eqref{Schur_Gunter} and \eqref{Schur_SUPG}, respectively. However, standard multigrid methods do not perform well for extremely anisotropic problems due to the stark difference in error frequency along versus across the field lines defining the direction of anisotropy. In our case, we found this approach to be particularly challenging for the SUPG-modified formulation leading to the system \eqref{matrix_SUPG}, likely due to additional asymmetry introduced by the SUPG-based modifications in the Schur complement \eqref{Schur_SUPG}.

Alternatively, one can follow the approach of \cite{wimmer2024fast}, where the system is considered as two weakly coupled directional derivatives, and multigrid methods specialized on transport such as AIR \cite{manteuffel2019nonsymmetric, manteuffel2018nonsymmetric} are applied instead. Such methods exploit the sparsity in stencil and matrix structure that arises due to upwinding, including of SUPG-type \cite{manteuffel2019nonsymmetric}. In future work, we aim to combine our SUPG-modified spatial discretization \eqref{def_cg_SUPG} with such an approach.
\end{remark}

%
\section{Numerical results} \label{sec_Numerical_results}
Having introduced our novel spatial discretization in Definition \eqref{def_cg_SUPG}, we next consider a series of numerical tests to validate the discretization's order of accuracy as well as to compare it against the existing CG methods \eqref{CG_straight} and \eqref{CG_mixed_Gunter} in scenarios relevant to magnetic confinement fusion simulations. As mentioned in the background section, we couple these spatial discretizations to an implicit mid-point rule in time. The tests are implemented in the finite element library Firedrake \cite{FiredrakeUserManual}, which heavily relies on PETSc \cite{balay2019petsc}. The system of equations are solved using the direct solver library MUMPS \cite{amestoy2000multifrontal}, and in the case of temperature-dependent conductivities, we use the nonlinear Newton solver functionality from PETSc, together with an exact Jacobian computed automatically in Firedrake. For each test case, we use the SUPG parameter $\tau$ as specified in \eqref{def_tau}. We will consider 2D quadrilateral and 3D prism meshes, using the standard continuous Galerkin space $P_k$ for our temperature space, and the discontinous Galerkin space $dP_{k-1}$ for $\zeta_h$ in the mixed discretization \eqref{CG_mixed_Gunter}. Our magnetic field $\mathbf{B}$ is set to be a function of a divergence conforming finite element space $\mathbb{V}_B$ of equal polynomial degree $k$ as our temperature CG space $P_k$. For the above types of meshes, $\mathbb{V}_B$ corresponds to the quadrilateral Raviart-Thomas space $RT_{c_k}^f$ and a suitable tensor product including the triangular Raviart-Thomas space $RT_k^f$, respectively \cite{mcrae2016automated}. There is no particular reason for this choice of $\mathbb{V}_B$ other than related work on magnetohydrodynamics by the authors, and in general we found similar results for other choices of magnetic field space.
\subsection{Convergence} First, we validate the SUPG-based discretization's order of accuracy using a simple 2D test case modeling dissipation of a Gaussian profile along magnetic field lines. The domain $\Omega$ is given by a unit square $[0, 1]^2$, and we set the magnetic field and initial temperature to
\begin{equation}
\mathbf{B} = (B_x, B_y)^T = [1, 0]^T, \hspace{1cm} T_0(x, y) =f(y) \; e^{-(x - 0.5)/\sigma^2}, \;\;\;\; f(y) =  \tfrac{1}{2}(1-\cos(2\pi y)), \label{Gaussian_analytic}
\end{equation}
for $\sigma = 0.2$. The conductivities are set to $\kappa_\parallel = 1$ and $\kappa_\perp =0.01$, and the boundary conditions are set equal to the values of the initial temperature field along the boundary. In the absence of perpendicular diffusion, the problem consists of decoupled 1D diffusion equations applied to Gaussian profiles on an interval $[0, 1]$, subject to Dirichlet boundary conditions. The solution to each 1D profile can be computed analytically using separation of variables and a Fourier decomposition given by $T_F = a_0 + \sum_{n=1}^m b_n \sin(\pi n x)e^{-\kappa_\parallel \pi^2t}$. We set $m = 200$ and use $f(y) T_F(\mathbf{x}, 0)$ for the initial projection into the initial discrete field $T_{h_0}$, and further set our forcing term equal to $S(t) = -\kappa_\perp \Delta \big(f(y)T_F\big)$. This way, the effects of perpendicular diffusion are balanced out, and $T$ evolves according to the aforementioned analytic solution. We can then compare our discrete solution $T_h$ against this analytic solution by considering the relative $L^2$ error
\begin{equation}
e(t) = \frac{\|T_h(t) - T_F(0)\|_2}{\|T_F(0)\|_2}, \label{relative_L2_error}
\end{equation}
where the finite element field $T_h$ and the Fourier mode expression $T_F$ (with pre-computed coefficients) are evaluated at quadrature points.

The polynomial degree for the temperature CG space is set to $k = 2$, and we consider a regular quadrilateral mesh for $\Omega$, with $10 \times 2^{k_x}$ cells in each coordinate direction, for $k_x \in \{1, 2, 3\}$. Further, the time step is set to $\delta t = 10^{-5}$ for each spatial resolution, and we run up to $t_{max} = 400 \delta t$. The resulting error convergence rates for the two standard primal and mixed as well as SUPG-based discretizations \eqref{CG_straight}, \eqref{CG_mixed_Gunter} and \eqref{SUPG_method}, respectively, are given in Figure \ref{analytic_errors}.
\begin{figure}[ht]
\begin{center}
\includegraphics[width=0.99\textwidth]{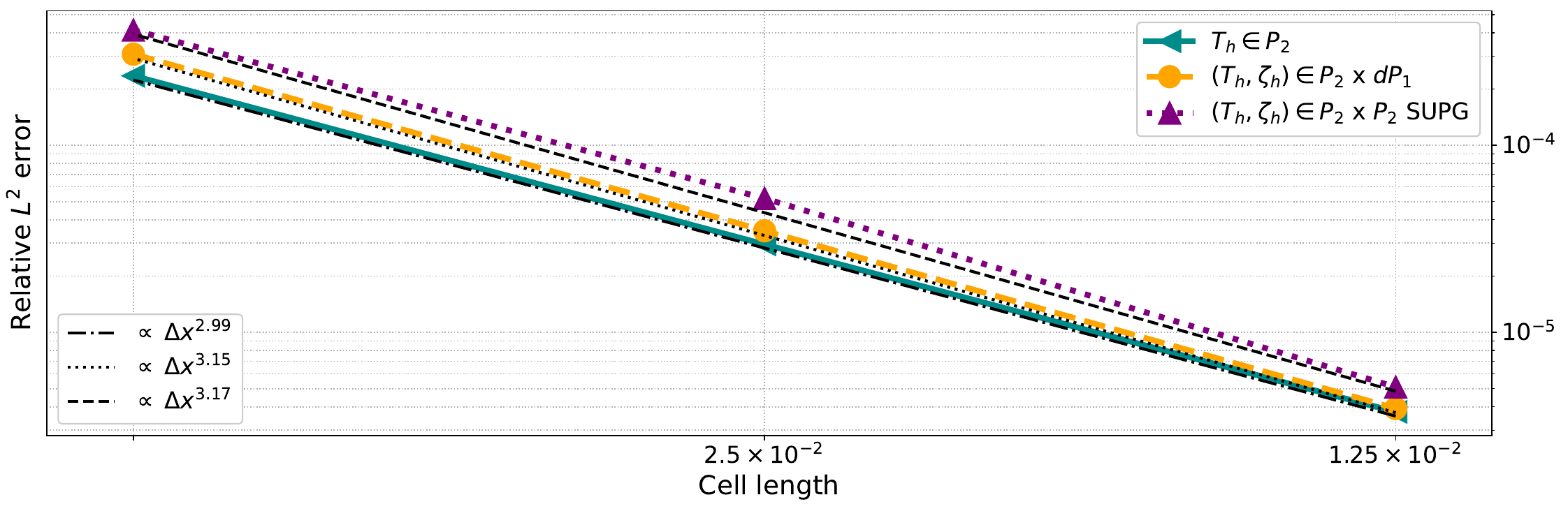}
\vspace{-2mm}
\caption{Convergence plot for relative $L^2$ errors \eqref{relative_L2_error} for analytic Gaussian profile diffusion test case \eqref{Gaussian_analytic}, for primal CG (solid cyan), standard mixed CG (dashed orange), and SUPG-based (dotted purple) formulations \eqref{CG_straight}, \eqref{CG_mixed_Gunter}, and \eqref{SUPG_method}, respectively.} \label{analytic_errors}
\end{center}
\end{figure}

For second order CG spaces, the expected convergence rate for the isotropic heat equation is given by 3, and all three CG discretizations approximately lead to such a rate for this test case. Further, as the degree of anisotropy is relatively small and the time step is small with respect to the parallel diffusion, the difference in error between the discretizations is small. When comparing the two mixed discretizations, we observe a slightly elevated error for the SUPG setup; in further tests, we found this to be due to using the CG space $P_2$ instead of the DG space $dP_1$ for the auxiliary variable. In other words, the latter DG space may be more suitable for representing the directional derivative $\mathbf{s} \cdot \nabla T$ than $P_2$; however, as we will see below, using $P_2$ leads to advantages in more practical test cases with higher anisotropy ratios, as it allows for the SUPG-stabilized approach detailed in the previous section.
\subsection{Magnetic flux surface} Having discussed our spatial discretization's order of accuracy, we next consider three tokamak simulation related test cases. The first represents a temperature perturbation on an idealized 2D magnetic flux surface (see Figure \ref{tokamak_setup}). For this purpose, we consider a 2D periodic domain $\Omega = [0, L_x] \times [0, L_y]$ for $(L_x, L_y) = (5, 4)$, with a magnetic field given by
\begin{equation}
\mathbf{B} = (1, m)^T,
\end{equation}
where $m = 20$ denotes the magnetic field's slope with respect to the surface. We then construct the temperature field perturbation in orthogonal coordinates $(\xi, \zeta)$, where the $\xi$ coordinate is aligned with the magnetic field lines, and the origin is placed at $(x, y) = (L_y/m, L_y/2)$. Further, the unit lengths for $\xi$ and $\zeta$ are defined such that the the $xy$-coordinate origin is at $(-2, -1)$ in $\xi\zeta$-coordinates. We then set the initial temperature equal to $T_0(\xi, \zeta) = T(\xi, \zeta, t)_{|_t=0}$, where
\begin{equation}
T(\xi, \zeta, t) = T_b + \left(1+ 4\kappa_\parallel \xi_0^2 t\right)^{-1/2}\; e^{-\frac{(\xi_0 \xi)^2}{\sqrt{1+4\kappa_\parallel \xi_0^2 t}}} \; (1 + \cos(\pi \zeta))/2, \label{flux_solution}
\end{equation}
if $\zeta \in (-1, 1)$, and $T = T_b$ otherwise. Further, $T_b = 0.2$, $\xi_0 = 2.5$ and we consider conductivities $\kappa_\parallel = 10$, $\kappa_\perp = 0$. Note that since there is no conductivity in the perpendicular direction, the problem decouples into a series of 1D Gaussians along separate magnetic field lines, and we obtain an analytic solution in time given by \eqref{flux_solution} together with its corresponding periodically shifted solutions with respect to our underlying domain $\Omega$ (up to the background temperature value $T_b$). For the discretization, we consider a pseudo-regular quadrilateral mesh with $120 \times 96$ cells, with inner vertices perturbed randomly by a factor of up to $0.1\delta x$ in order to avoid a decreased error due to an equal alignment of the magnetic field with the mesh throughout the domain. Additionally, we start with a small time step of $\delta t = 10^{-4}$, which is increased linearly to $\delta t = 0.1$ over 20 time steps. The simulation is run for the two standard primal and mixed spatial discretizations \eqref{CG_straight} and \eqref{CG_mixed_Gunter}, as well as the SUPG-based form \eqref{SUPG_method}. Results for the qualitative field development and qualitative as well as quantitative error development are shown in Figures \ref{flux_surface_T}, \ref{flux_surface_T_error} and \ref{flux_surface_error}, where we consider the relative $L^2$ error as defined in \eqref{relative_L2_error} (with $T_F$ replaced by the expression for $T$ from \eqref{flux_solution}, together with periodically shifted profiles).
\begin{figure}[ht]
\begin{center}
\includegraphics[width=0.99\textwidth]{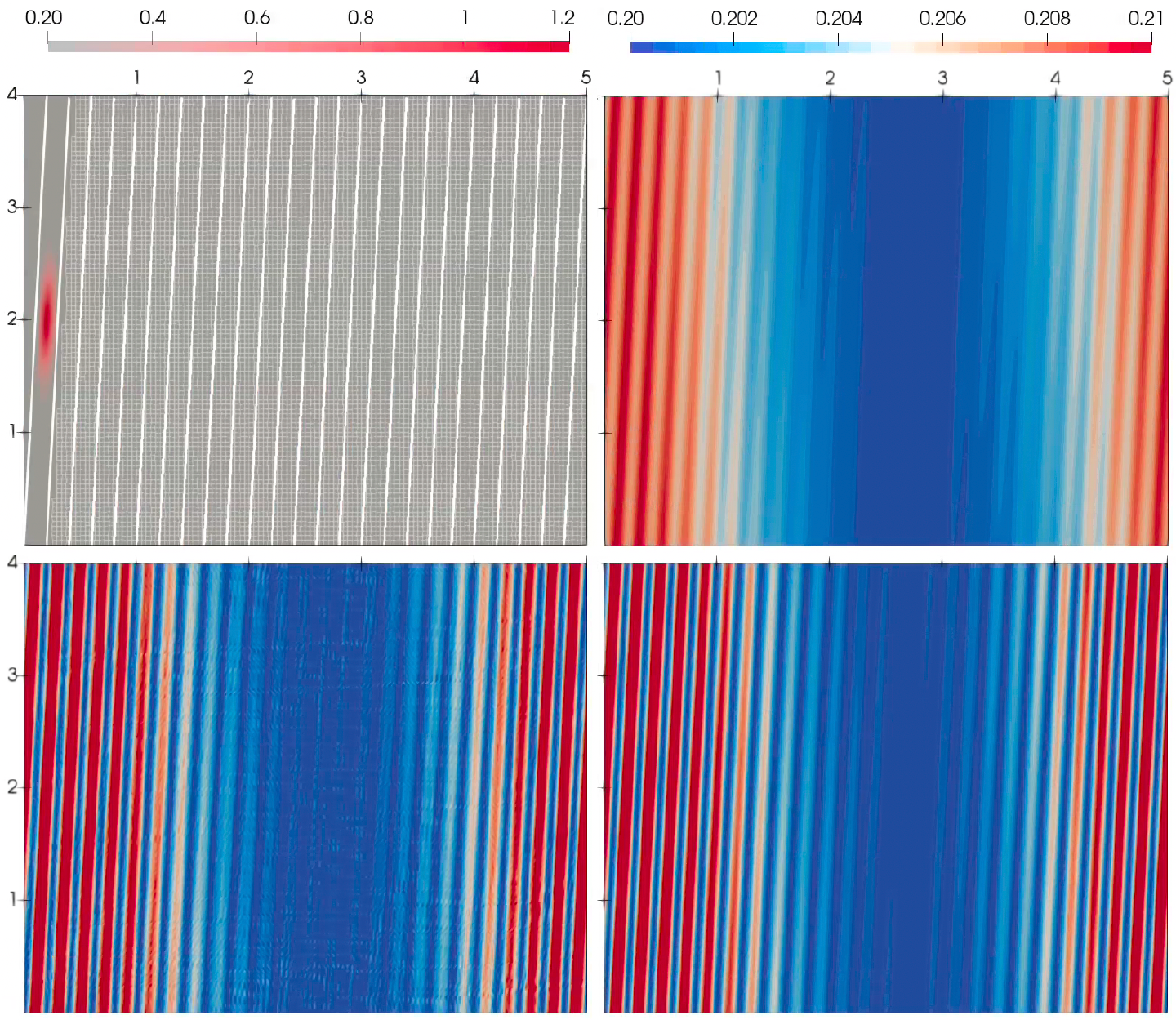}
\vspace{-2mm}
\caption{Images depicting temperature field $T$ for test case \eqref{flux_solution}. Top left: $T$ at initial time, together with straight white lines tracing a single (periodic) magnetic field line and a section of the quadrilateral mesh used for the simulation runs. Top right: primal CG formulation \eqref{CG_straight} at $t \approx 14$. Bottom row, left to right: standard mixed and SUPG-modified mixed formulations \eqref{CG_mixed_Gunter} and \eqref{SUPG_method}, respectively, at $t \approx 14$.} \label{flux_surface_T}
\end{center}
\end{figure}
\begin{figure}[ht]
\begin{center}
\includegraphics[width=0.99\textwidth]{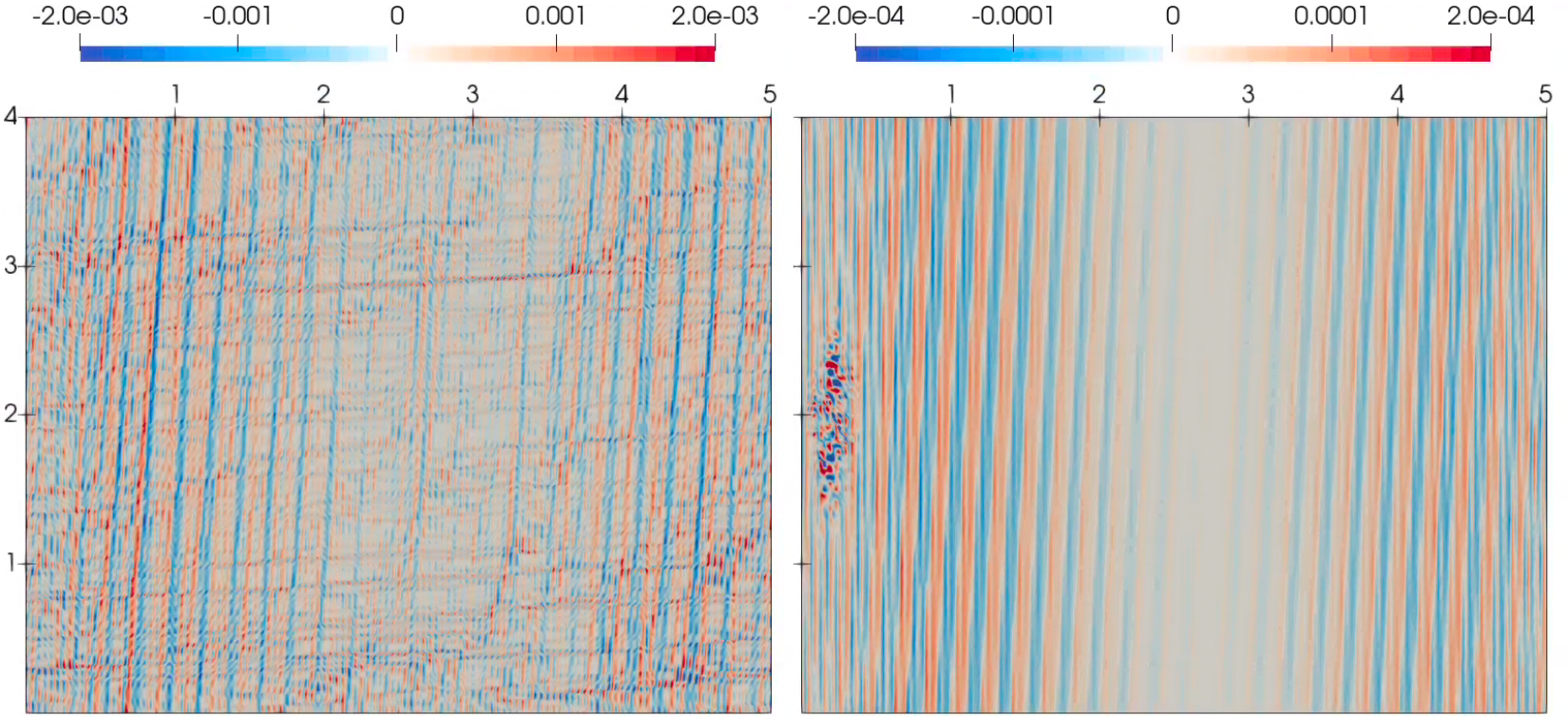}
\vspace{-2mm}
\caption{Images depicting error of temperature field $T$ for test case \eqref{flux_solution}. Left to right: standard mixed and SUPG-modified mixed formulations \eqref{CG_mixed_Gunter} and \eqref{SUPG_method} with color ranges of $\pm 2 \times 10^{-3}$ and $\pm 2 \times 10^{-4}$, respectively, at $t \approx 14$.} \label{flux_surface_T_error}
\end{center}
\end{figure}
\begin{figure}[ht]
\begin{center}
\includegraphics[width=0.99\textwidth]{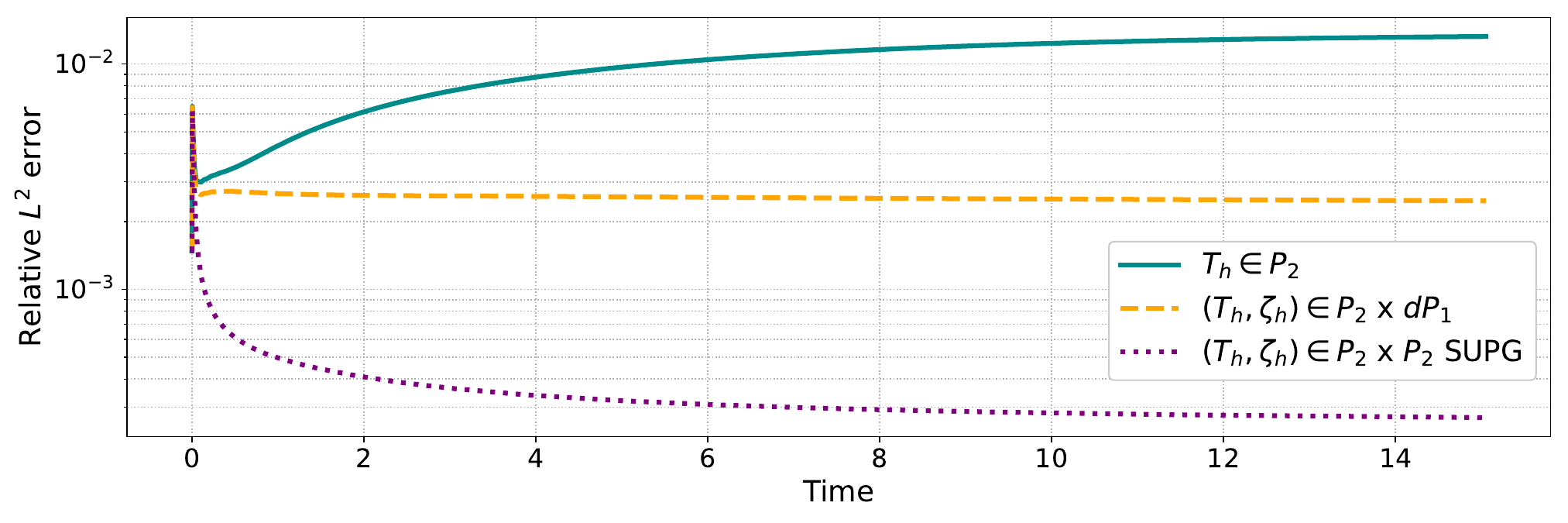}
\vspace{-2mm}
\caption{Relative $L^2$ error development over time for magnetic flux surface test case, for primal CG (solid cyan), standard mixed CG (dashed orange), and SUPG-based (dotted purple) formulations \eqref{CG_straight}, \eqref{CG_mixed_Gunter}, and \eqref{SUPG_method}, respectively.} \label{flux_surface_error}
\end{center}
\end{figure}

We find that the primal CG formulation exhibits a significant amount of spurious cross diffusion. This can be seen by considering the magnetic field line along $\zeta = 1$ (white straight lines in Figure \ref{flux_surface_T}), which initially does not contain any of the temperature perturbation. Since there is no perpendicular diffusion, we therefore expect $T = T_b$ along this field line throughout the simulation, which is clearly not the case for the primal CG formulation. Conversely, this is satisfied to a much better degree for both of the mixed formulations. However, the field development of the mixed formulations differs in that the standard mixed CG one leads to spurious grid scale noise (see Figure \ref{flux_surface_T_error}). In contrast, the SUPG-based formulation leads to the smallest error, which itself is dominated by noise near the perturbation's initial location. This noise is likely due to an initial condition projection error combined with a time stepping-related error in the perpendicular direction, which cannot be attenuated by the parallel diffusion. Overall, at the end of the simulation's runtime, we find relative $L^2$ errors of approximately $1.3\times10^{-2}$, $2.5 \times 10^{-3}$ and $2.7\times 10^{-4}$ for the primal CG, standard mixed CG, and SUPG-based mixed CG formulations, respectively, demonstrating a strong reduction in error through the auxiliary variable as well as the SUPG-stabilization.
\subsection{Tokamak equilibrium}
Next, we consider a full-torus tokamak scenario starting from an axisymmetric MHD equilibrium
\begin{equation}
\nabla p = \tfrac{1}{\mu_0}(\nabla \times \mathbf{B} ) \times \mathbf{B},
\end{equation}
for pressure $p$ and vacuum permeability $\mu_0$. The equilibrium is
generated by a Grad-Shafranov solver code \cite{liu2021parallel},
which returns $p$ inside the separatrix -- that is the boundary
confining the closed magnetic field lines; see Figure
\ref{tokamak_setup} -- and the components of $\mathbf{B}$ (in
cylindrical coordiantes $(R, \varphi, Z)$) on a regular rectangular
grid corresponding to a poloidal (i.e., $(R, Z)$) plane. The magnetic
field's safety factor -- that is the ratio of toroidal to poloidal
rotations of magnetic field lines -- varies between approximately $q_0=1.5$
at the magnetic axis and $q_{95}=6$ towards the separatrix. The equilibrium
resembles an ITER discharge, with a magnetic field
strength of $B_0 = 5.42$ T as well as a pressure of $p_0 = 656$ kPa,
both measured at the magnetic axis (see Figure \ref{tokamak_setup}).
We then interpolate the rectangular grid's field data into finite
element spaces defined on a 3D mesh. For faster computations, rather
than considering the entire torus, the mesh consists of a toroidally
(i.e. $\varphi$-directed) periodic section that is constructed by
periodically extruding an irregular triangular poloidal mesh in the
toroidal direction. The extrusion length is $\pi/4$ radians, and the
resulting 3D mesh consists of prisms aligned in the toroidal
direction. Note that the extrusion direction is curved, and we
therefore represent the mesh geometry by polynomials of degree equal
to that chosen for our temperature space. Finally, we note that in
practice, tokamak meshes are constructed to be aligned with the
flux surfaces. However, here we refrain from doing so since we are
interested in testing our new spatial discretization for scenarios in
which the magnetic flux surfaces are not aligned with the mesh. In
particular, while flux surface-aligned meshes can be constructed for
initial magnetohydrodynamic equilibria, this need not be the case at
later stages for more complex simulations involving e.g.,
magnetohydrodynamic instabilities that change the magnetic topology.

\begin{figure}[!htb]
\begin{center}
\includegraphics[width=0.99\textwidth]{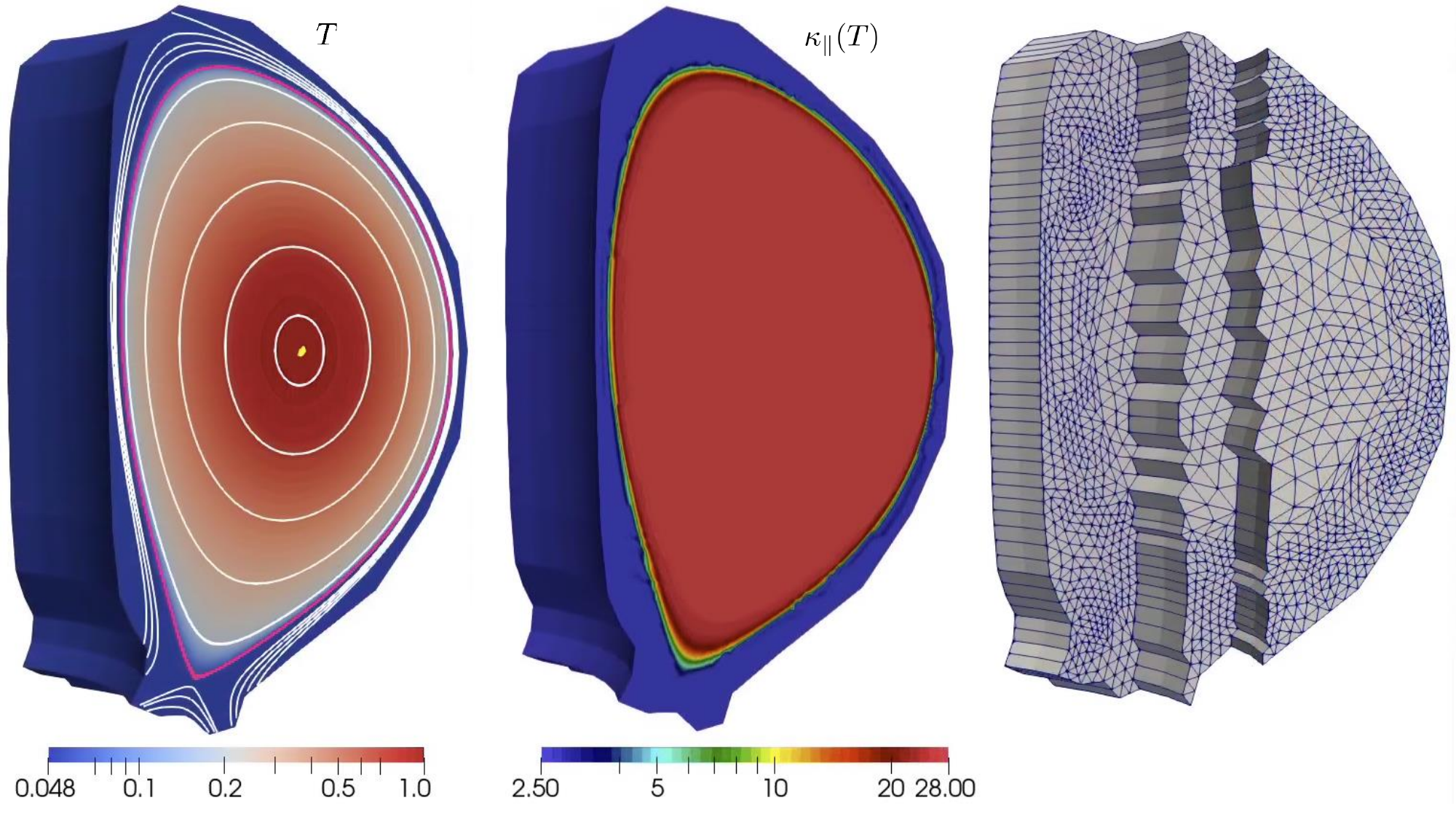}
\vspace{-2mm}
\caption{Images depicting toroidally periodic domain for tokamak test
  case setup. Left: non-dimensionalized initial temperature field $T$,
  and white lines denoting (vertical cross-sections of axisymmetric)
  magnetic flux surfaces. The non-dimensionalized magnetic field's
  $(R,Z)$-component is of magnitude at most approximately $0.3$, while
  its $\phi$-component's strength varies between approximately $0.75$
  and $1.5$. The purple curve and yellow dot denote the separatrix and
  magnetic axis, respectively. Center: initial parallel conductivity
  $\kappa_\parallel$ defined according to
  \eqref{k_par_tokamak}. Right: cross-section of once (poloidally)
  refined second order mesh created from toroidally periodic extrusion
  created from 2D irregular triangular poloidal
  mesh.} \label{tokamak_setup}
\end{center}
\end{figure}

The pressure is defined inside the separatrix only, noting that
$(\nabla \times \mathbf{B}) \times \mathbf{B} = \mathbf{0}$ outside
the separatrix. We add a floor value to $p$ throughout the domain,
which corresponds to a value found between the separatrix and the
tokamak chamber's wall. This value is typically very small, and for
the case of meshes not aligned with the separatrix, it may be smaller
than the values resulting from transferring the 2D regular poloidal
plane data onto the finite element grid for second or higher order
polynomial spaces, since the pressure field exhibits a very steep
gradient just inside the separatrix. This in turn may lead to negative
pressure and therefore temperature values in the simulation, causing
errors e.g. when evaluating the temperature-dependent parallel
conductivity \eqref{kappa_par_scale}. Since we do not use a
separatrix-aligned mesh in this test case, we resolve this by adding a
larger floor value, given by a factor of $p_b = 0.05 p_0$ of the
pressure at the magnetic axis.

Given $p$ and $\mathbf{B}$, we non-dimensionalize the problem and
domain with respect to their values $p_b + p_0$ and $B_0$ at the
magnetic axis, as well as the tokamak chamber's length scale. The
temperature $T$ and the corresponding boundary condition $T_{bc}$ is
computed from $p$ and $p_b$ through a representative density number
$n_0$ and the atomic number $Z$ of the plasma in
consideration. Further, we consider a scenario in which the
perpendicular diffusion is set to $\kappa_\perp = 0$, while the
parallel diffusion is set to
\begin{equation}
\kappa_\parallel =  8.8 \times 10^3 \; f(T)^{5/2}, \;\;\; f(T) = T_l - \sigma_l \ln\left(1 + e^{-\frac{T - T_l}{\sigma_l}}\right), \label{k_par_tokamak}
\end{equation}
for $T_l = 0.1$ and $\sigma_l = 0.04$. This choice follows the Braginskii transport model \cite{braginskii1965transport}, up to an upper limit reflecting restrictions to this model; for details on the non-dimensionalization and choice of $\kappa_\parallel$, see Appendix \ref{App_kappas}. In particular, for $T < T_l$, we have $f(T) \approx T$, while for $T > T_l$, we have $f(T) \approx T_l$. Since the initial axisymmetric pressure profile's contours are aligned with the magnetic flux surfaces, the absence of perpendicular diffusion implies an analytic solution of the anisotropic heat flux equation given by $T(\mathbf{x}, t) = T(\mathbf{x}, 0)$. For this reason, any change of the discretely computed temperature $T_h$ over time can be attributed to spurious cross-diffusion of temperature in the perpendicular direction\footnote{An additional error may be attributed to interpolating from the GS equilibrium solver grid onto the finite element grid; this can be avoided e.g., by using the same poloidal mesh in a finite-element based GS solver such as \cite{serino2024adaptive}.}.

We consider the spatial discretizations \eqref{CG_straight},
\eqref{CG_mixed_Gunter} and \eqref{SUPG_method} with CG spaces of
degree $k=2$, as well as $k=1$ in the case of \eqref{SUPG_method}. For
$k=2$, we use 3 cells in the toroidal extrusion, and for $k=1$, we use
a once refined version of the underlying 2D triangular poloidal mesh,
with 6 cells in the toroidal extrusion, leading to an overall equal
number of degrees of freedom as for the $k=2$ case. These two setups
then serve as the base resolution in space, which is refined in the
poloidal direction once -- noting that spurious cross-diffusion is
primarily a problem resulting from an inaccurate resolution of the
parallel heat flux' poloidal component -- for a total of 2 runs for
each choice of spatial discretization (see Figure \ref{tokamak_setup}
for the once refined mesh for $k=2$). To avoid oscillatory errors
related to initial adjustments of the discrete temperature profile to
the large parallel diffusion term, we start with a small time step of
$\delta t = 10^{-5}$, which is increased linearly to $\delta t = 5$
over 100 time steps. The simulation is run up to $t_{max} = 10^5$,
noting that after non-dimensionalization, the time unit is given by
the time taken for an Alfven wave to travel a characteristic distance
within the tokamak chamber. In particular, the choice of final time
step $\delta t$ and simulation runtime $t_{max}$ corresponds to values
that may be considered in (reduced) magnetohydrodynamic simulations
used to study slow, dissipative time scales of interest for resistive
magnetohydrodynamic instabilities
\cite{Biskamp_1993,park_plasma_1999,chacon2002implicit}. An important quantity with
respect to such simulations is the total amount of internal energy
stored within the plasma chamber; in our case for the
non-dimensionalized anisotropic heat equation, this can be translated
to $\|T(t)\|_1$. For the case of $\kappa_\perp = 0$, we expect this
quantity to remain constant. In practice after discretization,
spurious cross-diffusion will cause temperature to be diffused across
the separatrix and then ``removed'' from the domain via the fixed
small Dirichlet boundary conditions, noting that all magnetic field
lines outside the separatrix start and end at the domain's
boundary. In the worst case, all temperature beyond a constant
background profile $T \equiv T_{bc}$ is diffused. Time series of the
relative total available temperature
\begin{equation}
\frac{\|T(t)\|_1 - \|T_{bc}\|_1}{\|T(0)\|_1 - \|T_{bc}\|_1}, \label{rel_tot_t}
\end{equation}
for the different discretizations and spatial resolutions are depicted in Figure \ref{tokamak_losses}.

\begin{figure}[ht]
\begin{center}
\includegraphics[width=1.0\textwidth]{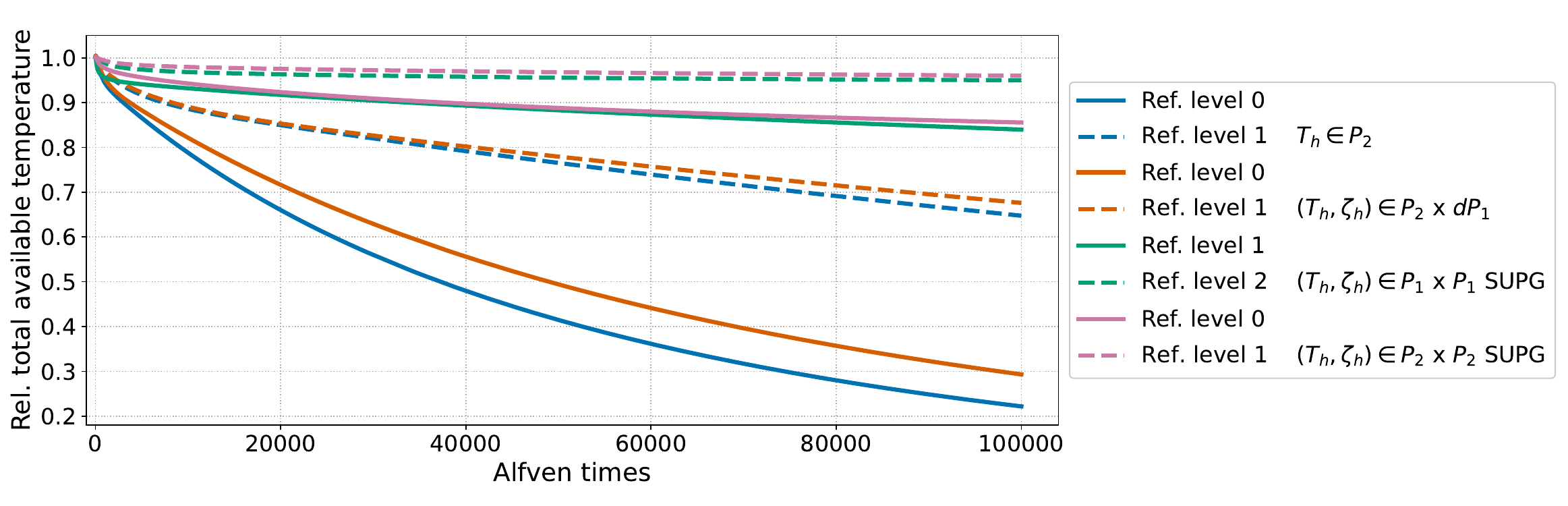}
\begin{tabular}{|c|c|c|c|c|} 
 \hline
$t = t_{max}$ & $T_h \in P_2$ & $(T_h, \zeta_h) \in P_2 \times dP_1$ & $(T_h, \zeta_h) \in P_1 \times P_1$ SUPG & $(T_h, \zeta_h) \in P_2 \times P_2$ SUPG\\
\hline
 lower ref. & 0.223 & 0.293 & 0.840 & 0.856\\
 higher ref. & 0.647 & 0.676 & 0.950 & 0.960\\
 \hline
\end{tabular}
\vspace{-1mm}
\caption{Evolution of the relative total available temperature \eqref{rel_tot_t} over time for tokamak equilibrium test case. Blue, brown, green and purple curves correspond to discretizations \eqref{CG_straight} with polynomial degree $k=2$, \eqref{CG_mixed_Gunter} with $k=2$, \eqref{SUPG_method} with $k=1$ and \eqref{SUPG_method} with $k=2$, respectively. The lower and higher refinement levels are denoted by solid and dashed lines, respectively. The table depicts the relative total available temperature at the last time step.}\label{tokamak_losses}
\end{center}
\end{figure}

We find that for the lower resolutions, the primal CG formulation \eqref{CG_straight} leads to a loss of approximately 78\% of the available total temperature, with some improvement when using the mixed formulation \eqref{CG_mixed_Gunter}, down to approximately 71\%. In contrast, the first and second order runs for the novel SUPG-based discretization \eqref{SUPG_method} lead to a significantly reduced loss at approximately 16\% and 14\%, respectively. In other words, comparing the second order standard mixed and mixed SUPG-modified methods, we obtain a 5-fold decrease in temperature loss. Here, we note that the parallel conductivity's temperature-dependence is an important mechanism that amplifies spurious perpendicular diffusion across the separatrix. This is because in the course of spurious cross-diffusion, the temperature in the region just outside the separatrix is raised, thereby locally increasing the parallel conductivity. This in turn leads to an increase in spurious cross-diffusion, thereby creating a cycle whose effect is stronger for methods with a less accurate representation of the directional derivative along magnetic field lines. In particular, we expect this mechanism to also hold true in the case of a more realistic, lower initial background temperature outside of the separatrix.

All four runs improve for the higher resolution, with a reduced loss of approximately 35\% and 32\% for the standard primal and mixed CG discretizations, respectively. However, this is still significantly larger than the losses reported for the SUPG-based discretizations run at the lower refinement level. For the higher refinement level, the SUPG-based runs lead to a loss of only approximately 5\% and 4\% respectively. This time, comparing the second order standard mixed and mixed SUPG-modified methods, we obtain an 8-fold decrease in temperature loss, from slightly less than a third down to a few percent. Further, in a run not shown here, the second order primal CG discretization run at the next higher (poloidal) refinement leads to a loss of approximately 17\%, which is still higher than the loss of the second order SUPG-based discretization run at the lowest refinement level (i.e. a quarter of the resolution).

Overall, the standard primal and mixed CG discretizations lead to losses that may not be deemed acceptable for magnetohydrodynamic simulations run for dissipative time scales, while the SUPG-modified discretizations -- especially the second order one -- lead to more reasonable losses as low as 4\%. Here, we note that the higher refinement level leads to approximately 30.5k degrees of freedom for the CG space $P_2$. Extrapolating to a full tokamak mesh (with $8\times3 = 24$ layers, i.e. 48 degrees of freedom in the toroidal direction), this corresponds to only $244k$ degrees of freedom for the temperature space. In particular, we expect a) an even better performance of our novel SUPG-based discretization in the case of higher resolutions and a mesh aligned with the initial conditions, and b) an improved performance in the case of more complex dynamical scenarios, given these coarse resolution results with a mesh not aligned with the magnetic flux surfaces, as well as the above magnetic flux surface results.
\begin{remark}[\textit{Higher order discretizations}]
Next to the first and second polynomial order runs shown in Figure \ref{tokamak_losses}, we also performed a run with the primal weak formulation \eqref{CG_straight} on the lower refinement mesh, with a CG function space that is second and fourth order in the toroidal and poloidal components, respectively. In particular, this leads to the same number of degrees of freedom as the higher refinement runs of Figure \ref{tokamak_losses}. In this case, we found a relative total available temperature of approximately $0.901$ at the last time step. In other words, the fourth order primal CG discretization leads to a loss of approximately 10\% in temperature, compared to only approximately 5\% for the lowest polynomial order mixed SUPG-based method at the same resolution.
\end{remark}
%
\subsection{Tokamak perturbation} Next to equilibrium scenarios for tokamak domains, we can also consider temperature perturbations that lead to a non-steady field evolution. For this purpose, we consider the same toroidally periodic extruded domain and magnetic field as for the equilibrium case. This time, however, we set the initial temperature equal to $T \equiv 0.5$ throughout, and apply a perturbation over time through an axisymmetric forcing term, given by
\begin{align}
&S(t) = S_p(R,Z)S_t(t), \\
&S_t(t) = \left(\tanh\big((t - t_s)/t_\sigma\big) +1\right)/2
&S_p =
\begin{cases}
s_m \big(1 + \cos(f_\rho \pi \rho)\big)/2 &\rho < 1/f_\rho, \\
0 &\rho > 1/f_\rho,
\end{cases} \label{pert_sexpr}
\end{align}
where $t_s = 40$, $t_\sigma = 15$ specify the forcing's growth over time, and $s_m = 1$, $f_\rho = 7$, $\rho = \sqrt{(R-R_0)^2 + (Z-Z_0)^2}$ specify its location in space, for $RZ$-coordinate $(R_0, Z_0)$. In our case, we set $(R_0, Z_0)$ to a particular mesh vertex location at approximately $(2.4, -0.22)$; see Figure \ref{tokamak_pert} (noting that we consider a tokamak domain that is non-dimensionalized with length scale $L_0 = 2$). The parallel conductivity is set to a fixed value of $\kappa_\parallel = 10$, and further $\kappa_\perp = 0$. As before, we start with a time step of $\delta t = 10^{-5}$, which this time is increased linearly to $\delta t = 0.02$ over 100 time steps. The simulation is run up to $t_{max} = 120$ Alfven times, using the mesh as depicted in Figure \ref{tokamak_setup}. Since there is no perpendicular diffusion, the temperature perturbation should be diffused strictly along magnetic field lines which form a magnetic flux tube, and in Figure \ref{tokamak_pert}, we marked two approximate bounding surfaces, which together form a volume that contains 95\% of the perturbation. In the course of the simulation, we expect the perturbation to steadily be diffused throughout the tube it lies in. Results for the standard mixed and SUPG discretizations  \eqref{CG_mixed_Gunter}, \eqref{SUPG_method}, respectively, are depicted in Figure \ref{tokamak_pert}, noting that the primal discretization \eqref{CG_straight} leads to a qualitatively similar result as \eqref{CG_mixed_Gunter}.
\begin{figure}[ht]
\begin{center}
\includegraphics[width=0.99\textwidth]{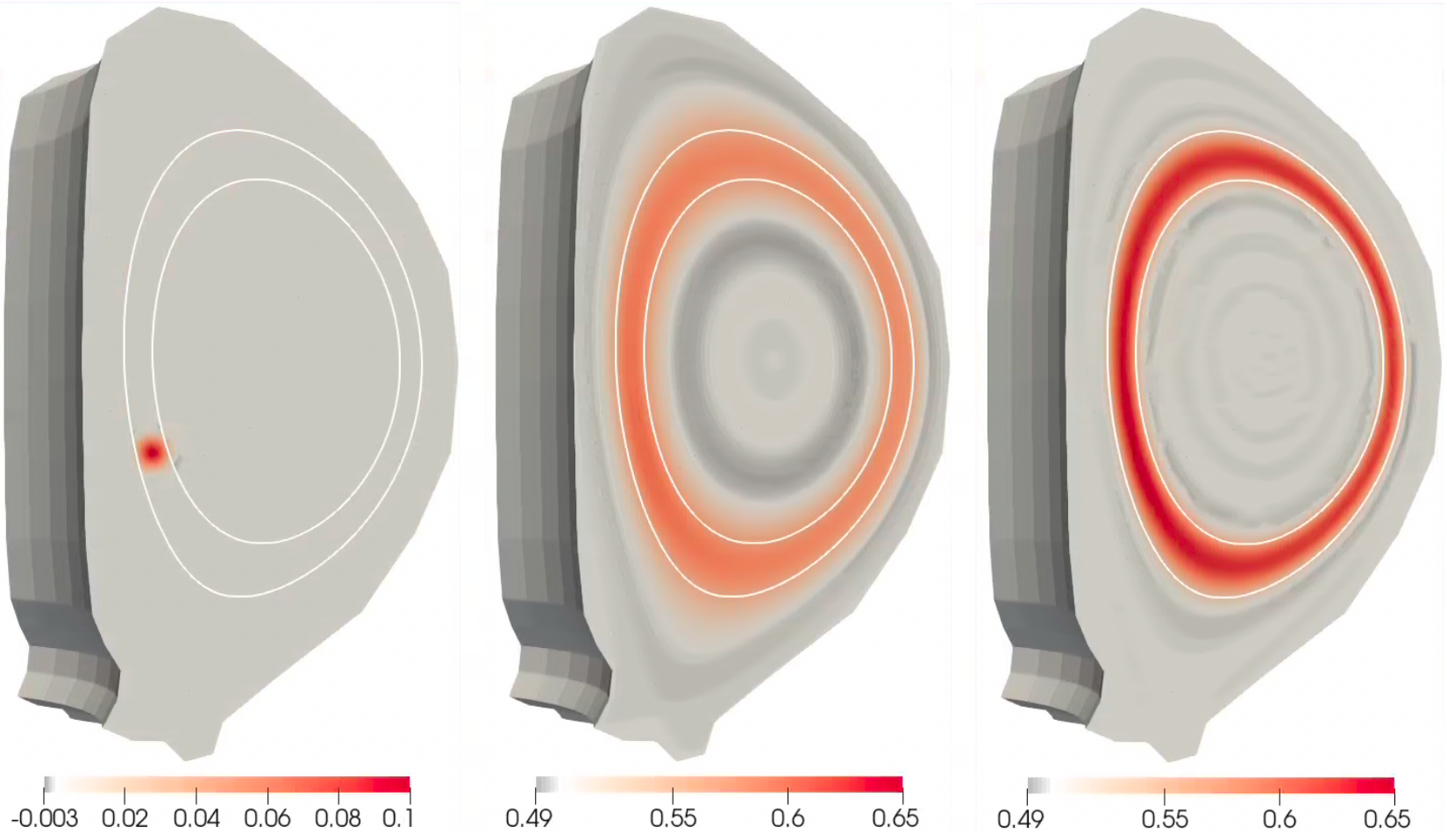}
\vspace{-2mm}
\caption{Images for perturbation test \eqref{pert_sexpr}, at $t \approx 120$ Alfven times. Left: forcing $S(t)$. Center and right: temperature field for standard mixed and SUPG discretizations  \eqref{CG_mixed_Gunter}, \eqref{SUPG_method}, respectively. White lines denote two magnetic flux surfaces forming a magnetic flux tube that contains 95\% of the temperature generated by $S(t)$.} \label{tokamak_pert}
\end{center}
\end{figure}

We find that as expected, the temperature perturbation generated by
the forcing term $S(t)$ is spread across the region encompassed by the
two aforementioned bounding flux surfaces. For the standard mixed
discretization \eqref{CG_mixed_Gunter}, we find a significant amount
of cross-diffusion, leaking temperature across the two bounding flux
surfaces. For the SUPG-based discretization, this is far less so the
case, and the temperature is well contained between the bounding
surfaces. In particular, both discretizations conserve the total temperature
up to boundary conditions\footnote{In the absence of boundary conditions, we can set $\gamma \equiv 1$ in \eqref{CG_mixed_Gunter_T} and \eqref{SUPG_method_T} and obtain $\tfrac{d}{dt}\int_\Omega T_h \; dx = \int_\Omega S \; dx$.}, and we find a maximum temperature of
approximately $0.65$ for the SUPG setup, compared to only $0.6$ for
the standard mixed one. Finally, both methods lead to areas with a
temperature less than the background value $0.5$; this is likely due
to small negative values in the discrete forcing term, and especially
the choice of relatively coarse time step and midpoint-rule for the
time discretization.
%
\section{Conclusion} \label{sec_conclusion}
In this work, we presented a novel CG-based spatial discretization for
the anisotropic heat conduction equation in the context of magnetic
confinement fusion. It contains an auxiliary variable that serves as a
form of the temperature's directional derivative, and further includes
an SUPG-stabilization method to reduce spurious heat flux across
magnetic field lines. We proved the method's consistency with respect
to strong solutions of the anisotropic heat flux equation, and further
motivated and discussed our choice of SUPG formulation, leading to a
parabolic character of the discretization up to consistency-related
modifications.

In numerical tests, we verified the spatial discretization's order of accuracy. Further, we considered a test case representing the spread of a temperature perturbation along a 2D magnetic flux surface with rational winding number, demonstrating a superior performance of our novel method when compared to existing CG-based discretizations. Finally, we considered two realistic tokamak scenarios representing a magnetohydrodynamic equilibrium and a perturbation spreading inside a magnetic flux tube, demonstrating a comparably far smaller spurious cooling rate and a more confined temperature spread, respectively, at time scales of interest for resistive magnetohydrodynamic simulations.

Possible future extensions to this work include studies in the context of anisotropic heat flux as part of a magnetohydrodynamic model, examining the discretization's accuracy when the temperature and magnetic fields are subject to fluid flow. Further, we are working on efficient solvers for our novel discretization in conductivity and time step regimes typically considered for tokamak simulations.

\subsection*{Acknowledgments}
This work was supported by the Laboratory Directed Research and Development program of Los Alamos National Laboratory, under project number 20240261ER, as well as the U.S. Department of Energy Office of Fusion Energy Sciences Base Theory Program, at Los Alamos National Laboratory under contract No. 89233218CNA000001. The computations have been performed using resources of the National Energy Research Scientific Computing Center (NERSC), a U.S. Department of Energy Office of Science User Facility operated under Contract No. DE-AC02-05CH11231. Further support was received by the Engineering and Physical Sciences Research Council (grant number EP/W029731/1). Los Alamos National Laboratory report number LA-UR-24-33208.
%
%
%
%
\appendix
\section{Non-dimensionalization} \label{App_kappas}
To derive a non-dimensionalized form of the anisotropic heat flux equation that is appropriate for the tokamak test cases considered in the numerical results section \ref{sec_Numerical_results}, we consider  parameter choices common for magnetohydrodynamic models. In such models the ion temperature equation may be given by \cite{jardin2010computational}
\begin{equation}
\frac{n_i}{\gamma - 1} \left(\pp{T_i}{t} + \mathbf{V}_i \cdot \nabla T_i\right) + n_i T_i \nabla \cdot \mathbf{V}_i - \nabla \cdot (\kappa_{i_\parallel} \nabla_\parallel T_i + \kappa_{i_\perp} \nabla_\perp T_i) = S, \label{p_eqn_full}
\end{equation}
for ion temperature $T_i$, ion density $n_i$, adiabatic index $\gamma = \tfrac{5}{3}$, ion velocity $\mathbf{V}_i$, ion conductivities $\kappa_{i_\parallel}$, $\kappa_{i_\perp}$, and forcing term $S$ containing additional contributions such as Ohmic heating, heat exchange, radiative cooling and external sources. The conductivites can be described by the Branginskii plasma model \cite{braginskii1965transport}, which leads to
\begin{equation}
\kappa_{i_\parallel} = \frac{3.9n_iT_i\tau_i}{m_i}, \hspace{1cm} \kappa_{i_\perp} =\frac{2 n_iT_i}{m_i \omega_i^2 \tau_i}. \label{kappa_braginskii}
\end{equation}
The additional quantities in \eqref{kappa_braginskii} denote the ion mass $m_i$, ion collisional time $\tau_i$ and ion gyrofrequency $\omega_i$, respectively. The latter two in turn are defined by
\begin{equation}
\tau_i = \frac{12\pi^{3/2}\epsilon_0^2 \sqrt{m_i} T^{3/2}}{n_i Z^4 e^4\ln\Lambda} \; \text{s}, \hspace{1cm} \omega_i = \frac{Z e|\mathbf{B}|}{m_i} \; \text{s}^{-1}, \label{tau_omega}
\end{equation}
for vacuum permittivity $\epsilon_0$, elementary charge $e$, Coulomb logarithm $\ln \Lambda$, and atomic number $Z$. For our numerical results in Section \ref{sec_Numerical_results}, we consider the following parameters and physical constants:
\begin{equation}
\begin{matrix*}[l]
&Z = 1, &\ln \Lambda \approx 15, &m_i =  1.673 \times 10^{-27} \; \text{kg}, \vspace{2mm}\\
&\epsilon_0 = 8.854 \times 10^{-12} \; \text{Fm}^{-1}, &e = 1.6 \times 10^{-19} \; \text{C}, &\mu_0 = 4\pi \times 10^{-7} \text{N A}^{-2},
\end{matrix*}
\end{equation}
where $\mu_0$ denotes the vacuum permeability and will be used further below.

We non-dimensionalize \eqref{p_eqn_full} with respect to typical MHD
equilibrium quantities in tokamaks, using a plasma chamber with a
prescribed minor radius $L_0$ and a typical ion density number
$n_0$. Further, we consider a magnetic field strength and plasma
pressure which are given by $B_0$ and $p_0$, respectively, at the
magnetic axis, as computed by a Grad-Shafranov equilibrium code. For
the equilibrium used in the tokamak test case in Section
\ref{sec_Numerical_results}, the quantities are given by
\begin{equation}
p_0 = 656 \; \text{kPa}, \;\;\; B_0 = 5.42 \; \text{T}, \;\;\; n_0 = 10^{20} \; \text{m}^{-3}, \;\;\; L_0 = 2\; \text{m}, \;\;\; t_A = 1.83 \times 10^{-7} \; \text{s}, \label{quantities_non_dim}
\end{equation}
where $t_A = L_0/v_A$ denotes the characteristic Alfven wave time
scale, for Alfven wave speed $v_A = B_0/\sqrt{\mu_0n_0m_i}$. These are
representative values for an ITER discharge. In the following, we will
assume for simplicity charge neutrality and further the ion and
electron temperature to be approximately equal, so that the total
plasma pressure at the magnetic axis can be written in terms of the ion
temperature $T_0$ according to
\begin{equation}
p_0 \approx (1 + Z)n_0 T_0 = 2n_0 T_0 \;\;\; \implies \;\;\; T_0 = 3.28 \times 10^{-15} \; \text{J},
\end{equation}
which corresponds to approximately $20.5$ keV at magnetic axis. Note
that as in the formulas above, $T_0$ is given in Joule (with $T_0 =
k_B T_K$, for Boltzmann constant $k_B$ and temperature in Kelvin
$T_K$). Given these choices, the above quantities \eqref{tau_omega},
\eqref{kappa_braginskii} evaluate to \begingroup
\addtolength{\jot}{2mm}
\begin{subequations}
\begin{align}
&\tau_i = \frac{12\pi^{3/2}\epsilon_0^2 \sqrt{m_i} T_0^{3/2}}{n_0 Z^4 e^4\ln\Lambda} \tilde{n}^{-1}\tilde{T}^{3/2} = \tau_{i_0} \tilde{n}^{-1}\tilde{T}^{3/2}, & \omega_i = \frac{Z e|\mathbf{B}_0|}{m_i} |\tilde{\mathbf{B}}| = \omega_{i_0} |\tilde{\mathbf{B}}|, \\
&\kappa_{i_\parallel} = \frac{3.9n_0T_0\tau_{i_0}}{m_i} \tilde{T}^{3/2} = \kappa_{i_{\parallel 0}}\tilde{T}^{3/2}, & \kappa_{i_\perp} =\frac{2 n_0T_0}{m_i \omega_{i_0}^2 \tau_{i_0}} \tilde{n}^2 \tilde{T}^{-1/2} |\tilde{\mathbf{B}}|^{-2} = \kappa_{i_{\perp0}},
\end{align}
\end{subequations}
\endgroup
where tildes denote non-dimensionalized quantities. Further, the system's characteristic collisional time, gyrofrequency and conductivities, which depend on physical constants as well as our choice of $Z$, $\ln \Lambda$, $B_0$, $T_0$ and $n_0$, evaluate to
\begingroup
\addtolength{\jot}{2mm}
\begin{subequations}
\begin{align}
&\tau_i \approx 4.1 \times 10^{-2} \; \text{s}, & \omega_i \approx 5.2 \times 10^8 \; \text{s}^{-1}, \\
&n_0^{-1} \kappa_{i_{\parallel 0}} \approx 3.13 \times 10^{11} \; \text{m}^2\text{s}^{-1}, & n_0^{-1} \kappa_{i_{\perp 0}} \approx 3.56 \times 10^{-4} \; \text{m}^2\text{s}^{-1}.
\end{align}
\end{subequations}
\endgroup
Altogether, the non-dimensionalized form of the ion temperature equation \eqref{p_eqn_full} reads
\begin{equation}
\frac{\tilde{n} n_0}{\gamma - 1} \frac{T_0}{t_A}\pp{\tilde{T}}{\tilde{t}}  - \frac{1}{L_0} \tilde{\nabla} \cdot \left(\kappa_{i_{\parallel 0}}  \; \tilde{T}^{5/2} \frac{1}{L_0} \tilde{\nabla}_{\parallel} (T_0\tilde{T}) +  \kappa_{i_{\perp 0}}  \; \tilde{n}^2 |\tilde{\mathbf{B}}|^{-2} \tilde{T}^{-1/2} \frac{1}{L_0} \tilde{\nabla}_{\perp} (T_0\tilde{T})\right) = \tilde{S}, \label{T_eqn_orig_nondim}
\end{equation}
noting that $\mathbf{b}$ in the directional gradients is transformed to $\mathbf{b} = \tilde{\mathbf{B}}/|\tilde{\mathbf{B}}|$. Multiplying \eqref{T_eqn_orig_nondim} by $(\gamma - 1) t_A/(T_0n_0)$, we then obtain
\begin{align}
\tilde{n} \pp{\tilde{T}}{\tilde{t}} - \tilde{\nabla} \cdot \left( K_\parallel \tilde{\kappa}_\parallel \; \tilde{\nabla}_\parallel \tilde{T} + K_\perp \tilde{\kappa}_\perp \;\tilde{\nabla}_\perp \tilde{T} \right) = \frac{(\gamma - 1)t_A}{T_0n_0} \tilde{S},
\end{align}
for conductivity parameters
\begingroup
\addtolength{\jot}{2mm}
\begin{align}
&K_\parallel = (\gamma - 1)\frac{t_A}{L_0^2 n_0/\kappa_{i_{\parallel_0}}} = (\gamma - 1)\frac{t_A}{t_\parallel}, & \tilde{\kappa}_\parallel = \tilde{T}^{5/2}, \hspace{23mm} \\
&K_\perp = (\gamma - 1)\frac{t_A}{L_0^2 n_0/\kappa_{i_{\perp 0}}} = (\gamma - 1)\frac{t_A}{t_\perp}, & \tilde{\kappa}_\perp = \tilde{n}^2 |\tilde{\mathbf{B}}|^{-2} \tilde{T}^{-1/2}, \hspace{10mm}
\end{align}
\endgroup
where $t_{\parallel}$ and $t_{\perp}$ denote diffusive time scales. In this work, we assume for simplicity a constant density field $\tilde{n} = 1$, leading to the non-dimensionalized anisotropic heat flux equation \eqref{T_eqn_orig}, with $\kappa_\parallel = K_\parallel \tilde{\kappa}_\parallel$ and $\kappa_\perp = K_\perp \tilde{\kappa}_\perp$.

For our choice of spacial and temporal scales $L_0$, $t_A$, respectively, from \eqref{quantities_non_dim}, we obtain
\begin{equation}\label{eq:kappa-high}
K_\parallel \approx 8.8 \times 10^3, \;\;\; K_\perp \approx 1 \times 10^{-11},
\end{equation}
and these will be our conductivity magnitudes at the magnetic axis, since our non-dimensionalization implies $\tilde{T} = |\tilde{\mathbf{B}}| = 1$ here. Conversely, we note that the smallest temperature values occur in the region between the separatrix and the plasma facing wall. A typical value in this region is given by $T_b = 10$ eV, corresponding to a factor of $T_b = 4.9 \times 10^{-4}$ of $T_0$. Additionally, towards the outer wall region (with respect to the radial coordinate $r$), we have $|\tilde{\mathbf{B}}| \approx 0.75$. Correspondingly, in this region, the conductivity magnitudes are given by
\begin{equation}
K_\parallel \tilde{\kappa}_\parallel = K_\parallel T_b^{5/2} \approx 4.7 \times 10^{-5}, \hspace{1cm} K_\perp \tilde{\kappa}_\perp = K_\perp |\tilde{\mathbf{B}}|^{-2} T_b^{-1/2} \approx 8 \times 10^{-10},
\end{equation}
which together with \eqref{eq:kappa-high} implies a parallel to perpendicular anisotropy ratio between approximately 5 and 15 orders of magnitude.

The underlying assumptions for the Braginskii
closure~\cite{braginskii1965transport} of parallel thermal conduction
do not hold in the high plasma temperature limit where the magnetic
field line length $L_c$ is shorter than the mean free path
$\lambda_\textrm{mfp}.$ This is when the formal perturbative expansion
of the distribution function in Braginskii's analysis no longer
applies because the Knudsen number $K\!n\equiv \lambda_\textrm{mfp}/L_c$
is large.  A straightforward scaling analysis of Braginskii's parallel
thermal conduction flux for electrons has $q_{e\parallel} =
\kappa_\parallel \nabla_\parallel T_e \sim n_e v_{th,e} T_e K\!n,$ with
$n_e$ the electron density, $v_{th,e}$ the electron thermal speed, and
$T_e$ the electron temperature.  In the collisional regime $K\!n \ll 1.$
When $K\!n$ reaches order unity or greater, the normal expectation is that
parallel heat flux would be bounded by electron thermal speed,
yielding the so-called flux limiting form of $q_{e\parallel} \sim n_e
v_{th,e} T_e.$~\cite{Bell-pof-1985} In a range of tokamak plasma
problems, more exotic parallel electron conduction physics can
appear~\cite{PhysRevLett.108.165005, Li_2023, Zhang_2023}. They all have
the effect of suppressing the parallel heat flux far below the
Braginskii values at high plasma temperature.  For our purposes, this
limiting behavior is mimicked by deploying the limiting function $f$
in \eqref{k_par_tokamak} in the numerical results section. With the
limiter in place, the maximum anisotropy ratio reduces to
approximately 12 orders of magnitude, which in practice may be further
reduced since with the aforementioned limitations, the perpendicular
conductivity may also be larger than predicted by the Braginskii model
because of plasma micro-turbulence. In general, the limiting case for
particle transport is complex, and our choice of limit $T_l$ in
\eqref{k_par_tokamak} is somewhat arbitrary and set roughly to the
temperature near the separatrix, noting that the separatrix marks a
transition zone from long mean free paths and high temperatures to
shorter mean free paths and low temperatures. In results not shown
here, we also performed a comparison without the use of a limiter
(i.e., such that $f(T) = T$ in \eqref{k_par_tokamak}), and found that
the standard primal and mixed CG discretizations' accuracy greatly
deteriorates. This is much less so the case for the SUPG-based one,
leading to a more favorable comparison than the one presented in the
numerical results section.

Finally, we note that the time scales of interest in MHD simulations
for tokamaks are typically inferred from dissipative dynamics,
including $t_{\kappa_\perp}$ and the plasma current decay time scale
$t_R$ due to resistivity. After non-dimensionalization, the latter
quantity is formulated through the so-called Lundquist number $S,$ which
denotes the ratio between $t_A$ and $t_R$. Depending on the underlying
physics assumptions for resistivity, $S$ may lie in the range $10^6$
-- $10^{12}$ \cite{jardin2010computational, park_plasma_1999}. Considering the lower end
of this range, we therefore have that given a non-dimensionalized unit
time scale of $1$ Alfven time, simulations of interest may be run up
to $t_{max} = 10^6$ Alfven times. In this case, the tokamak test case
in Section \ref{sec_Numerical_results} with its runtime of $t_{max} =
10^5$ Alfven times therefore corresponds to a simulation up to many growth times of
the resistive modes \cite{park_plasma_1999}.

%
%
\bibliographystyle{plain}
\bibliography{Anisotropic_CG_SUPG}
\end{document}